\documentclass[11pt]{amsart}
\usepackage{fullpage}
\usepackage{indentfirst}
\usepackage{mathtools}
\usepackage{amsmath}
\usepackage{amsthm}
\usepackage{amssymb}
\usepackage{float}
\usepackage[english]{babel}
\usepackage{fullpage}
\usepackage{hyperref}
\usepackage{tikz}
\usepackage{caption}

\hypersetup{
    colorlinks,
    linkcolor={red!80!black},
    citecolor={red!80!black},
    urlcolor={blue!80!black}
}

\usepackage{scalerel}
\usetikzlibrary{patterns}
\usetikzlibrary{decorations.markings, decorations.pathreplacing}

\tikzstyle{vtx}=[circle, draw, fill=black, inner sep=0pt, minimum width=2pt]
\tikzstyle{every node}=[circle, draw, fill=black!50, inner sep=0pt, minimum width=4pt]
\tikzstyle{text}=[circle, draw, fill=blue!50, inner sep=0pt, minimum width=4pt]

\tikzset{middlearrow/.style={
		decoration={markings,
			mark= at position 0.5 with {\arrow[very thick]{#1}} ,
		},
		postaction={decorate}
	}
}

\newtheorem{theorem}{Theorem}[section]
\newtheorem*{theorem*}{Theorem}

\newtheorem{lemma}[theorem]{Lemma}
\newtheorem{question}[theorem]{Question}

\newtheorem{corollary}[theorem]{Corollary}
\newtheorem{conjecture}[theorem]{Conjecture}
\newtheorem{claim}[theorem]{Claim}
\newtheorem*{claim*}{Claim}
\newtheorem{definition}[theorem]{Definition}

\newtheorem{prop}[theorem]{Proposition}

\newtheorem{remark}[theorem]{Remark}

\def\N{\mathbb{N}}
\newcommand{\CP}{\mathcal{P}}
\newcommand{\CT}{\mathcal{T}}
\newcommand{\G}{\mathcal{G}}

\newcommand{\wt}{\widetilde}

\newcommand{\C}{\mathcal{C}}
\newcommand{\D}{\mathcal{D}}
\newcommand{\veps}{\varepsilon}

\newenvironment{claimproof}[1]{\par\noindent\textit{Proof of the claim:}\space#1}{\leavevmode\unskip\penalty9999 \hbox{}\nobreak\hfill\quad\hbox{$\blacksquare$}}

\title{On the Tur\'anability and tileability of oriented graphs}
\author{Igor Araujo}
\author{Zimu Xiang}
\thanks{IA: Research supported in part by the James E. and Rebecca A. Wetzel Fellowship. \\
\indent ZX: Research supported in part by the Franz Hohn and J.P. Nash Fellowship.\\
\indent Both authors supported in part by NSF RTG DMS-1937241.}
\address{Department of Mathematics, University of Illinois Urbana-Champaign, Urbana, IL 61801, USA} 
\email{\parbox[t]{\linewidth}{\{igoraa2, zimux2\}@illinois.edu}.} 

\date{}

\begin{document}

\begin{abstract}
    An oriented graph $H$ is \emph{Tur\'anable} (resp. \emph{tileable}) if there exist $n_0 \in \mathbb{N}$ such that every semi-regular near-tournament on $n \ge n_0$ vertices contains a copy of $H$ (resp. a perfect $H$-tiling). We disprove a conjectured characterization of Tur\'anable oriented graphs by DeBiasio, Han, Lo, Molla, Piga, and Treglown~\cite{DHLMPT}
    , show that there are Tur\'anable oriented graphs which are not tileable, and provide a new example of tileable oriented graph. 
\end{abstract}

\maketitle

\section{Introduction}

For graphs $H$ and $G$, an \emph{$H$-tiling} in $G$ is a collection of disjoint copies of $H$ inside $G$. A \emph{perfect $H$-tiling} is an $H$-tiling that covers all vertices of $G$. Minimum degree thresholds to ensure perfect tilings in graphs have been extensively studied since the 1960s (see, e.g., \cite{corradi} for perfect $K_3$-tiling, and \cite{hajnal} for perfect $K_r$-tiling), culminating in the celebrated result by K\"uhn and Osthus~\cite{kuhn-osthus} which provides the correct minimum degree threshold, up to an additive term, for perfect $H$-tiling for every graph $H$. 

In this paper, we are interested in the analog version of this problem for oriented graphs.
A \emph{directed graph}, or digraph, is a a set of vertices and directed edges, i.e. oriented pairs of distinct vertices. An \emph{oriented graph} is a direct graph for which for every pair of vertices $u$ and $v$ at most one of the ordered edges $uv$ or $vu$ is included. A \emph{tournament} is an oriented graph for which for every pair of vertices $u$ and $v$, exactly one of the ordered edges $uv$ or $vu$ is included.
An oriented graph is a \emph{near-tournament} if for every vertex $u$, there is at most one vertex $v$ disjoint from $u$ such that neither $uv$ or $vu$ is included.

For oriented graphs, a natural analog of minimal degree is the minimum semi-degree. The \emph{minimum semi-degree} $\delta^0(G)$ of an oriented graph $G$ is the minimum of all in- and out-degrees of the vertices of $G$. 
Note that the minimum semi-degree of an oriented graph on $n$ vertices is at most $\lfloor \frac{n-1}{2} \rfloor$. An oriented graph $G$ on $n$ vertices for which $\delta^0(G) = \lfloor \frac{n-1}{2} \rfloor$ is called a \emph{semi-regular near-tournament}. When $n$ is odd, a semi-regular near-tournament on $n$ vertices is called a \emph{regular tournament}.
It is natural to seek minimum semi-degree thresholds for perfect $H$-tiling for oriented graphs $H$.

\begin{question}\label{quest:min_degree_tiling}
    For an oriented graph $H$ on $h$ vertices, and $n \in \N$ a multiple of $h$, what is the smallest number $f(n, H)$ so that every oriented graph $G$ on $n$ vertices with $\delta^0(G) \ge f(n,H)$ contains a perfect $H$-tiling?
\end{question}

This question has a different flavor than the analog question for graphs, since for some oriented graphs $H$ there are regular tournaments avoiding even a single copy of $H$ (see, e.g., Theorem~\ref{thm:bh}).
This motivates the following definitions.

\begin{definition}[Tur\'anable] \rm
    We say that an oriented graph $H$ is \emph{Tur\'anable} if there exists $n_0 \in \N$ such that every semi-regular near-tournament on $n \ge n_0$ vertices contains a copy of $H$.
\end{definition}

\begin{definition}[Tileable] \rm
    We say that an oriented graph $H$ on $h$ vertices is \emph{tileable} if there exists $n_0 \in \N$ such that for every semi-regular near-tournament $G$ on $n \ge n_0$ vertices where $h$ divides $n$, $G$ contains a perfect $H$-tiling.
\end{definition}

With this notation, tileable oriented graphs are exactly the oriented graphs $H$ for which Question~\ref{quest:min_degree_tiling} is well-defined for sufficiently large $n$.
One could also ask an analogous equivalent to Question~\ref{quest:min_degree_tiling} for the containment of a single copy of $H$. 

\begin{question}\label{quest:min_degree_turan}
    For an oriented graph $H$ and $n \in \N$, what is the smallest number $g(n, H)$ such that every oriented graph $G$ on $n$ vertices with $\delta^0(G) \ge g(n,H)$ contains a copy of $H$?
\end{question}

We will consider the asymptotic normalized version of the minimum degree threshold given in Question~\ref{quest:min_degree_turan}. That is, we define 
\begin{equation} \label{eq:kappa}
    \kappa^0(H) \coloneqq \lim_{n \to \infty} \frac{g(n,H)}{n} .
\end{equation}
As noted in~\cite{DHLMPT}, this limit can be shown to exist with a proof similar to~\cite[Proposition 1.2]{mubayi-zhao}.

The goal of this paper is to begin the systematic study of Tur\'anable and tileable oriented graphs, aiming at finding a characterization for both classes. This might seem a bold task at first, but we note that similar results in the context of edge-ordered graphs are already known~\cite{APTX, GMNPTV}. Our first result disproves a conjectured characterization of Tur\'anable oriented graphs by DeBiasio, Han, Lo, Molla, Piga, Treglown~\cite{DHLMPT}. Before we state their conjecture, we need the following definition.

For integers $a, b, c \ge 1$, let $D_{a,b,c}$ denote the tournament on $a+b+c$ vertices obtained from the blow-up of the consistently oriented triangle $C_3$ by replacing its vertices with transitive tournaments on $a$, $b$, and $c$ vertices.
When $a=b=c=s$, we write $D_s$ to denote $D_{s,s,s}$, i.e., $D_s$ is the tournament on $3s$ vertices obtained from the $s$-blow-up of the consistently oriented triangle $C_3$ by replacing the three independent sets with transitive tournaments.

\begin{conjecture}[Conjecture 8.3 in~\cite{DHLMPT}] \label{old_conj}
    An oriented graph $H$ is Tur\'anable if and only if there is an $s \in \N$ with $H \subset D_s$.
\end{conjecture}

\begin{theorem} \label{thm:conj_is_false}
    There exists a Tur\'anable oriented graph which is not a subgraph of $D_s$ for any $s \in \N$. 
\end{theorem}

We propose an updated conjecture on a characterization of Tur\'anable oriented graphs. In order to state the conjecture, we need to introduce the following notation. 
First, let $C_{k}^{\ell}$ denote $\ell$-th power of a consistently oriented cycle on $k$ vertices.
Furthermore, let $F_r$ denote the tournament with vertex set $V(F_r) = [3]^r$ where there is an edge from $u=(u_1,\dots, u_r)$ to $v = (v_1, \dots, v_r)$ if $(u_i,v_i) \in \{(1,2),(2,3),(3,1)\}$ for $i \in [r]$ being the smallest such that $u_i \neq v_i$. 

Analogously, $F_r$ can be defined iteratively as $F_1 = C_3$ and, for every $r\ge 1$, $F_{r+1}$ consists of three vertex-disjoint copies of $F_r$, say $V_1$, $V_2$, $V_3$ such that all edges between distinct copies go from $V_1$ to $V_2$, from $V_2$ to $V_3$, and from $V_3$ to $V_1$.
Note that $F_r$ is a regular tournament on $3^r$ vertices, while $C_{2k+1}^k$ is a regular tournament on $2k+1$ vertices. Conjecture~\ref{old_conj} was motivated by the following result.

\begin{theorem}[Bollob\'as, H\"{a}ggkvist~\cite{bollobas}] \label{thm:bh}
    A tournament $T$ is Tur\'anable if and only if there is an $s \in \N$ with $T \subset D_s$.\footnote{For a complete proof of Theorem~\ref{thm:bh}, see Section 8.2 in~\cite{DHLMPT}.}
\end{theorem}

However, the proof of Theorem~\ref{thm:bh} in~\cite{bollobas} explores the fact that, for a tournament $T$, there exist $r$ and $k$ such that $T \subset F_r$ and $T \subset C_{2k+1}^k$ if and only if there exists an $s$ such that $T \subset D_s$. We then suggest the following updated conjecture.

\begin{conjecture} \label{new_conj}
    An oriented graph $H$ is Tur\'anable if and only if there are integers $r, k$ such that $H \subset F_r$ and $H \subset C_{2k+1}^k$.
\end{conjecture}

We note that Conjecture~\ref{new_conj} is similar to Theorems~2.4 and~2.6 in~\cite{APTX}, in which the regular tournaments $F_r$ and $C_{2k+1}^k$ would be the analog to the \emph{canonical edge-orderings} in the context of oriented graphs.

The second goal of this paper is to study the differences between Tur\'anable and tileable oriented graphs. In particular, DeBiasio, Han, Lo, Molla, Piga, and Treglown asked the following question.

\begin{question}[Question 8.4 in~\cite{DHLMPT}] \label{quest:dhlmpt} 
    Is it true that an oriented graph $H$ is tileable if and only if $H$ is Tur\'anable?
\end{question}

We answer this question in the negative. It follows from Theorem~\ref{thm:bh} that the oriented graph $D_s$ is Tur\'anable for every $s \ge 1$. Our second result is the following.

\begin{theorem}\label{thm:ds} 
    For every $s\ge 2$, the oriented graph $D_s$ is not tileable.
\end{theorem} 

Next, we explore quantitative bounds for Question~\ref{quest:min_degree_turan}. Recall that we denote by $C_{k}^\ell$ the $\ell$-th power of the consistently oriented cycle on $k$ vertices. For $\kappa^0$ as defined in~\eqref{eq:kappa}, the authors of~\cite{DHLMPT} noted that $\kappa^0(C_{3t}^t) \ge \lfloor \frac{3t-2}{2} \rfloor \cdot \frac{1}{3t-1}$ follows from the fact that the blow-up of a semi-regular near-tournament on $3t-1$ vertices does not contain a copy of $C_{3t}^t$. They then asked whether or not this is the correct value for $\kappa^0(C_6^2)$.

\begin{question}[Question 8.7 in~\cite{DHLMPT}] 
    Is it true that $\kappa^0(C_6^2) = 2/5$?
\end{question}

We show that this is not the case.

\begin{prop} \label{prop:kappa}
    $\kappa^0(C_6^2) \ge 3/7$.
\end{prop}

Our knowledge of tileability is much shorter than the one on Tur\'anability: The only graphs known to be tileable are subgraphs of transitive tournaments~\cite{yuster} and consistently oriented cycles~\cite{li-molla, wang}.
Our final result, and the main technical portion of this paper, is dedicated to adding one more example to the list of known tileable oriented graphs.

\begin{theorem} \label{thm:d112}
    The oriented graph $D_{1,1,2}$ is tileable.
\end{theorem}

The remainder of the paper is organized as follows. We prove Theorem~\ref{thm:conj_is_false} in Section~\ref{sec:conj_is_false}, Theorem~\ref{thm:ds} in Section~\ref{sec:ds}, and Proposition~\ref{prop:kappa} in Section~\ref{sec:kappa}. Section~\ref{sec:d112} is devoted to proving Theorem~\ref{thm:d112}. Finally, we have some concluding remarks in Section~\ref{sec:conclusion}.

\subsection*{Notation} 
We write $[n] \coloneqq \{1,2,\dots, n\}$ for the set of natural numbers up to $n$.

Let $G$ be an oriented graph. We write $V(G)$ for its vertex set and $E(G)$ for its edge set. For $u, v \in V(G)$ we simply write $uv$ to denote the oriented edge $\overrightarrow{uv}$. We also write $uvw$ to denote the consistently oriented triangle, that is, with edges $uv$, $vw$, and $wu$. For $X \subseteq V(G)$, we denote by $G[X]$ the oriented graph induced by the set $X$. 
We denote by $N^+(v)$ the out-neighborhood of $v$, and write $N^+(v,X)=N^+(v)\cap X$ and $d^+(v,X)=|N^+(v,X)|$.
When $X=V(G)$, we may simply write $d^+(v)=d^+(v,V(G))$.
We define $N^-(v)$ as the in-neighborhood of $v$, and $N^-(v,X),d^-(v,X),d^-(v)$ similarly.
We write $N(v)=N^+(v)\cup N^-(v)$ for the neighborhood of $v$, and subsequently $d(v)=|N(v)|=d^+(v)+d^-(v)$ for the total degree of $v$.
Given a partition $\CP=(V_1,\dots,V_d)$ of $V(G)$ and a subset $U\subseteq V(G)$, the \emph{index vector} of $U$ with respect to $\CP$, denoted by $\mathbf{i}_\CP(U)$, is defined as $\mathbf{i}_\CP(U)=(|U\cap V_1|,\dots,|U\cap V_d|)$.

\section{A Tur\'anable oriented graph not contained in $D_s$} \label{sec:conj_is_false}

In this section, we will prove Theorem~\ref{thm:conj_is_false} by providing an example of an oriented graph that is Tur\'anable but not a subgraph of $D_s$ for any $s\ge 1$. We will denote by $S$ the following oriented graph on 5 vertices (see Figure~\ref{fig:graph_S}). Its vertex set is $V(S) \coloneqq \{u_1, u_2, u_3, u_4, u_5 \}$ and edge set 
\begin{equation}
    E(S) \coloneqq \{ u_1u_2, u_1u_3, u_2u_3, u_2u_4, u_3u_4, u_3u_5, u_4u_1, u_5u_2 \} .
\end{equation}

\begin{figure}
    \begin{minipage}{0.45\textwidth}
        \centering
        \begin{tikzpicture}[thick,scale=0.8] 
            \foreach \y in {1,...,5} {
                \draw[middlearrow={latex reversed}, blue] (72*\y : 3) -- (72*\y + 144 : 3);
        		}
            \draw[middlearrow={latex reversed}, blue] (0 : 3) -- (72 : 3);
            \draw[middlearrow={latex reversed}, blue] (72 : 3) -- (144 : 3);
            \draw[middlearrow={latex reversed}, blue] (144 : 3) -- (216 : 3);
        
            \draw (72*1 : 3) node [label=above:{$\scaleto{u_{3}}{8pt}$}] {};
            \draw (72*2 : 3) node [label=left:{$\scaleto{u_{2}}{8pt}$}] {};
            \draw (72*3 : 3) node [label=left:{$\scaleto{u_{1}}{8pt}$}] {};
            \draw (72*4 : 3) node [label=below:{$\scaleto{u_{5}}{8pt}$}] {};
            \draw (72*5 : 3) node [label=right:{$\scaleto{u_{4}}{8pt}$}] {};
        \end{tikzpicture}
    \end{minipage}
    \begin{minipage}{0.45\textwidth}
        \centering
        \begin{tikzpicture}[thick,scale=1] 
            \foreach \y in {1,2,3} {
                \draw[middlearrow={latex reversed}, blue] (120*\y+90 : 1) -- (120*\y-30 : 1);
                }
            \draw[middlearrow={latex}, blue] (2,-3) -- (210 : 1);
            \draw[middlearrow={latex}, blue] (2,-3) -- (330 : 1);
            \draw[middlearrow={latex}, blue] (210 : 1) -- (-2,-3);
            \draw[middlearrow={latex}, blue] (330 : 1) -- (-2,-3);
            \draw[middlearrow={latex}, blue] (-2,-3) -- (2,-3);
            
            \draw (90 : 1) node [label=above:{$\scaleto{u_{5}}{8pt}$}] {};
            \draw (210 : 1) node [label=left:{$\scaleto{u_{2}}{8pt}$}] {};
            \draw (330 : 1) node [label=right:{$\scaleto{u_{3}}{8pt}$}] {};
            \draw (-2,-3) node [label=below:{$\scaleto{u_{4}}{8pt}$}] {};
            \draw (2,-3) node [label=below:{$\scaleto{u_{1}}{8pt}$}] {};
        \end{tikzpicture}
    \end{minipage}
    \caption{Different drawings of the oriented graph $S$. Note that $S \subset C_5^2$, $S \subset F_2$, and $D_{1,1,2} \subset S$.}
    \label{fig:graph_S}
\end{figure}
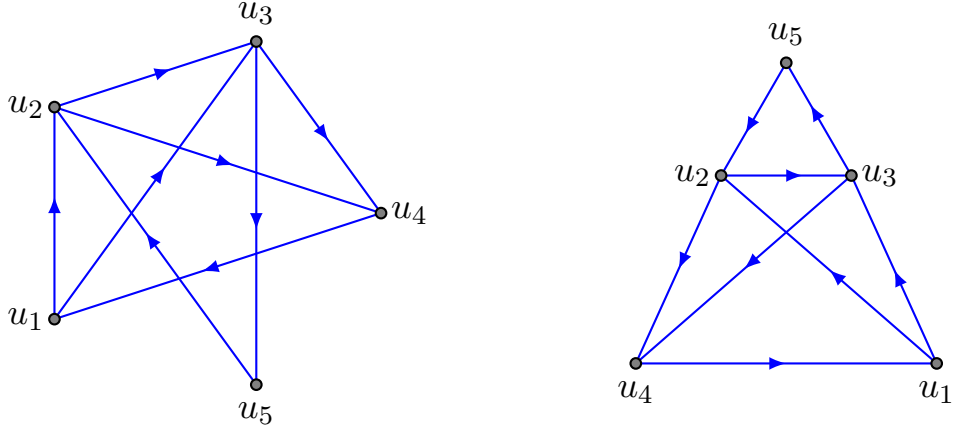

The remainder of the section is devoted to showing that the oriented graph $S$ defined above is a counterexample to Conjecture~\ref{old_conj}.

\begin{prop} \label{prop:S} 
    The oriented graph $S$ is Tur\'anable and is not a subgraph of $D_s$ for any $s \in \N$.
\end{prop}

We begin with the following claim.

\begin{claim} \label{claim:S_not_in_Ds}
    The graph $S$ is not a subgraph of $D_s$ for any $s \in \N$.
\end{claim}

\begin{claimproof}
    Assume by contradiction that $\varphi: V(S) \to V(D_s)$ is an embedding of $S$ into $D_s$ for some $s \in \N$. Let $v_i \coloneqq \varphi(u_i)$ for every $i \in [5]$. Let $V_1$, $V_2$, $V_3$ be the three parts of $D_s$ that induce transitive tournaments so that $v_1 \in V_1$ (and with edges pointing from $V_1$ to $V_2$). The only consistently oriented triangles in $D_s$ are those using vertices in three distinct parts of $D_s$, therefore, since $u_1u_2u_4$ and $u_2u_3u_5$ are consistently oriented triangles in $S$, we conclude that $v_1\in V_1$, $v_2 \in V_2$, and $v_3 \in V_3$. This leads to a contradiction, as the edge $v_1v_3$ cannot be from part $V_1$ to $V_3$. 
\end{claimproof}

To conclude Proposition~\ref{prop:S} we will prove the following.

\begin{claim} \label{claim:S_turanable}
    The graph $S$ is Tur\'anable.
\end{claim} 

\begin{claimproof}
    Let $n$ be sufficiently large and $G$ be a semi-regular near-tournament on $n$ vertices. We will assume for the sake of contradiction that $G$ does not contain a copy of $S$.
    Recall that from Theorem~\ref{thm:bh} it follows that $D_{1,1,3}$ is Tur\'anable.
    Hence, $D_{1,1,3} \subset G$. Consider a fixed copy $F$ of $D_{1,1,3}$ inside $G$ where $V(F)=\{u_0, u_1, u_2, u_3, u_4\}$ with parts $V_1 = \{u_1\}, V_2 = \{u_2, u_0, u_3\}, V_3 = \{u_4\}$, where the edges inside $V_2$ are $u_2u_0, u_0u_3, u_2u_3$, and the edges between distinct parts are oriented from $V_1$ to $V_2$, from $V_2$ to $V_3$, and from $V_3$ to $V_1$.
    
    Let $A\coloneqq V(G) \setminus V(F)$. For every vertex $v \in A$, if the edges connecting $v$ and $u_2, u_3$ have orientation $vu_2$ and $u_3v$ then $G[\{u_1, u_2, u_3, u_4, v\}]$ contains a copy of $S$ (see Figure~\ref{fig:graph_S}), a contradiction. Therefore, $N^-(u_2, A) \cap N^+(u_3, A) = \emptyset$.
        
    Note that since $G$ is a semi-regular near-tournament, then
    $$d^-(u_2) = 1 + d^-(u_2, A) \ge \frac{n-2}{2} \qquad \text{and} \qquad 
    d^+(u_3) = 1 + d^+(u_3, A) \ge \frac{n-2}{2} ,$$
    implying that
    $$ |N^-(u_2, A) \cup N^+(u_3, A)| = d^-(u_2, A) + d^+(u_3, A) \ge n-4 ,$$
    contradicting the fact $|N^-(u_2, A) \cup N^+(u_3, A)| \le |A| = n-5 < n-4$.
\end{claimproof}

\section{The (non)tileability of $D_s$} \label{sec:ds}

In this section, we will prove Theorem~\ref{thm:ds} by constructing, for each $s\ge 2$, a family of semi-regular near-tournaments which do not contain a perfect $D_s$-tiling. 
We begin by exhibiting a divisibility barrier for consistently oriented triangle tiling. 
Recall from Question~\ref{quest:min_degree_tiling} that we denote by $f(n,H)$ the minimum semi-degree threshold for an oriented graph to contain a perfect $H$-tiling.
We show that, for odd $n$, the minimum semi-degree threshold for perfect $C_3$-tiling is as large as possible, that is, $f(n,C_3) \ge \frac{n-1}{2}$. 

Let $G$ be a tournament on $3n$ vertices for which there is a partition $(V_1, V_2, V_3)$ of its vertices where $|V_1|=n-1$, $|V_2|=n$, and $|V_3|=n+1$ and all edges between different parts are oriented from $V_1$ to $V_2$, $V_2$ to $V_3$, and $V_3$ to $V_1$. 
We claim that $G$ does not contain a perfect $C_3$-tiling. Indeed, every copy of $C_3$ in $G$ is contained in a part $V_i$ or has exactly one vertex in each part. Therefore, if $\CT$ is a collection of vertex-disjoint copies of $C_3$ in $G$, then 
$|V_1 \cap V(\CT)| \equiv |V_2 \cap V(\CT)| \equiv |V_3 \cap V(\CT)| \mod 3$ 
and $\CT$ cannot be a perfect $C_3$-tiling.
Note that, by choosing the orientation of edges inside each part to correspond to semi-regular near-tournaments, we have $\delta^0(G) = \frac{3n-3}{2}$, which implies $f(3n,C_3)>\frac{3n-3}{2}$. 

The above example is almost enough to show that $C_3$ is not tileable, in the sense that we need to swap the orientation of only a few edges to make the tournament $G$ semi-regular. The obstacle is that it is not possible to reverse the orientation of the edges in $G$ and keep the property that there is no perfect $C_3$-tiling (indeed, $C_3$ is tileable~\cite{li-molla}). However, this is not an issue when we are dealing with perfect $D_s$-tilings for $s\ge 2$. By inverting the orientation of few edges in $G$ we can turn it into a semi-regular near-tournament and keep its perfect $D_s$-tiling freeness. Formally, we consider the following tournament.

For every $s\ge 2$, and every $k \ge 0$ so that $s(k+1)$ is even, we define the tournament $T=T(s,k)$ as follows. There is a partition $(V_1, V_2, V_3)$ of its vertices where $|V_1|=s(k+1)-1$, $|V_2|=s(k+1)$, and $|V_3|=s(k+1)+1$ and the induced tournament $T_i=T[V_i]$ is a semi-regular near-tournament for every $i \in [3]$. 
The edges between $V_1$ and $V_2$ are oriented from $V_1$ to $V_2$. 
The edges between $V_2$ and $V_3$ that are not oriented from $V_2$ to $V_3$ form a matching $M$ so that $$|V(M)\cap V_2|=|V(M)\cap V_3|=\frac{s(k+1)}{2},$$ and $V(M)\cap V_2$ is the set of vertices in $V_2$ with the highest out-degree in $T_2=T[V_2]$.
The edges between $V_1$ and $V_3$ are oriented from $V_3$ to $V_1$.

\begin{claim}
    For every $s\ge 2$, and every $k \ge 0$ so that $s(k+1)$ is even, the tournament $T(s,k)$ is a semi-regular near-tournament on $3s(k+1)$ vertices.
\end{claim}

\begin{claimproof}
    Let $s\ge 2$, $k \ge 0$ be such that $s(k+1)$ is even and set $T=T(s,k)$. It is enough to prove that every vertex $v \in V(T)$ has out-degree $d^+(v) \in \left\{ \frac{3s(k+1)}{2}-1, \frac{3s(k+1)}{2} \right\}$. 
    For $v_1 \in V_1$, we have $d^+(v_1, V_1) = \frac{s(k+1)-2}{2}$, $d^+(v_1, V_2) = s(k+1)$, and $d^+(v_1, V_3)=0$.
    For $v_3 \in V_3$, we have $d^+(v_3, V_3) = \frac{s(k+1)}{2} $, $d^+(v_3, V_1) = s(k+1)-1$, and $d^+(v_3, V_2) \le 1$.
    For $v_2 \in V_2$, we have $d^+(v_2, V_1) = 0$.
    If $v_2 \in V(M)$, then $d^+(v_2, V_3) = s(k+1)$ and $d^+(v_2, V_2) = \frac{s(k+1)}{2}$.
    If $v_2 \not\in V(M)$, then $d^+(v_2, V_3) = s(k+1)+1$ and $d^+(v_2, V_2) = \frac{s(k+1)-2}{2}$.
    In any case, we conclude that $d^+(v) = d^+(v, V_1)+d^+(v, V_2)+d^+(v, V_3) \in \left\{ \frac{3s(k+1)}{2}-1, \frac{3s(k+1)}{2} \right\}$, as desired.
\end{claimproof}

\medskip
Theorem~\ref{thm:ds} follows from the following.

\begin{prop}\label{prop:dr}
    For every $s\ge 2$, and every $k \ge 0$ so that $s(k+1)$ is even, the tournament $T(s,k)$ does not contain a perfect $D_s$-tiling.
\end{prop} 

We note that when $s(k+1)$ is odd, a similar construction also yields a regular tournament on $3s(k+1)$ vertices without perfect $D_s$-tiling. For simplicity, we omit this construction. We will show that Proposition~\ref{prop:dr} holds for a larger family of oriented graphs, which includes $T(s,k)$, but not all oriented graphs in the family are semi-regular near-tournaments. 

Given an oriented graph $G$ with a partition $\CP=(V_1,V_2,V_3)$ of its vertex set, we refer to the edges going from $V_1$ to $V_3$, $V_3$ to $V_2$, or $V_2$ to $V_1$ as \emph{reverse edges}.
Denote $G_\CP$ as the subgraph of the underlying graph of $G$, with the edge set containing only the reverse edges (without orientation) with regard to $\CP$.
We define the family $\G$ as the family of all oriented graphs for which there is a partition $\CP=(V_1,V_2,V_3)$ of its vertices such that $G_\CP$ is a matching.
Note that $T(s,k)$ belongs to $\G$ and the matching $M$ corresponds to the reverse edges of $T(s,k)$.

We begin by introducing a lemma that exploits symmetry to simplify subsequent arguments.

\begin{lemma}\label{lem:sym}
    Let $G\in \G$ and $s\ge 2$. For every permutation $\pi \in S_3$, the following holds. If $\CP=(V_1,V_2,V_3)$ is an ordered partition of $V(G)$ and $\varphi: V(D_s) \to V(G)$ is an embedding of $D_s$ into $G$ with $\mathbf{i}_\CP(\varphi(D_s))=(a_1,a_2,a_3)$, then there is some $G'\in\G$ with an ordered partition $\CP' = (V_1',V_2',V_3')$ of $V(G')$ and an embedding $\varphi': V(D_s) \to V(G')$ of $D_s$ into $G'$ with $\mathbf{i}_{\CP'}(\varphi'(D_s))=(a_{\pi(1)},a_{\pi(2)},a_{\pi(3)})$.
\end{lemma}

\begin{proof}
    First assume that $\pi(1)=2$, $\pi(2)=3$, and $\pi(3)=1$. Then we let $G'=G$, $\CP'=(V_1',V_2',V_3')$ with $V_i'=V_{\pi(i)}$ for $i\in [3]$ and $\varphi'=\varphi$.
    By definition, $G_\CP$ contains the same edge set as $G_{\CP'}$, and it follows that $\mathbf{i}_{\CP'}(\varphi'(D_s))=(a_{\pi(1)},a_{\pi(2)},a_{\pi(3)})$.
    
    Thus, we assume that $\pi(1)=1$, $\pi(2)=3$, and $\pi(3)=2$.
    Now let $G'$ be the oriented graph with all the orientations reversed as compared to $G$, a partition $\CP' = (V_1',V_2',V_3')$ with $V_1'=V_1$, $V_2'=V_3$ and $V_3'=V_2$, and finally $\varphi'=\varphi$. Note that the oriented graph obtained by reversing the orientations of all edges of $D_s$ is isomorphic to $D_s$, hence $\varphi'$ is an embedding of a copy of $D_s$ into $G'$.
    By construction, $\mathbf{i}_{\CP'}(\varphi'(D_s))=(a_{\pi(1)},a_{\pi(2)},a_{\pi(3)})$.

    Since the two permutations above generate the group $S_3$, the result follows.
\end{proof}

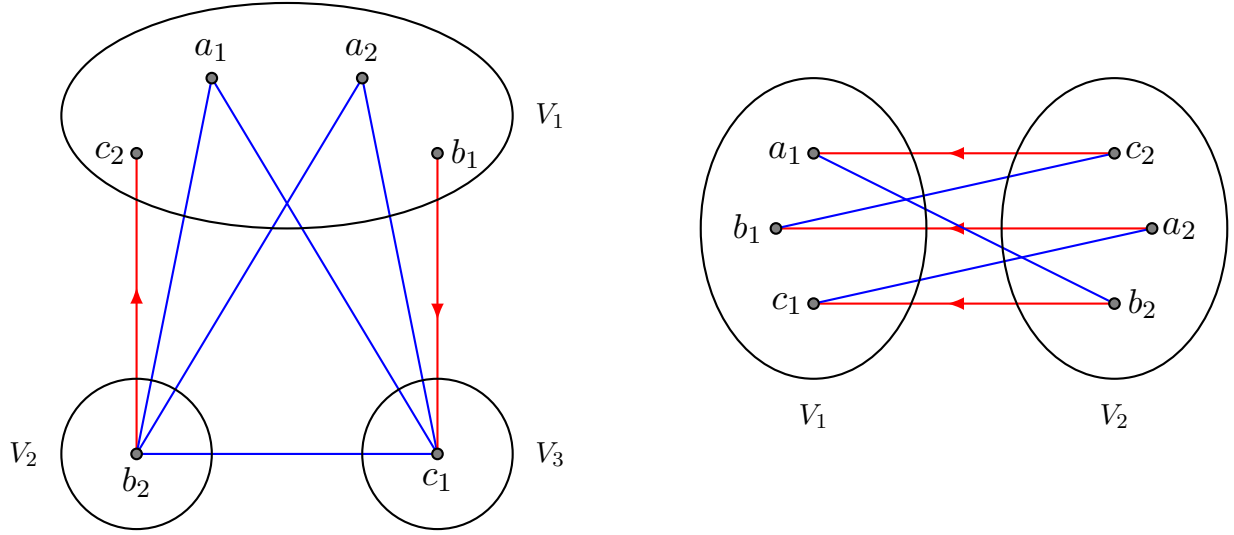
\begin{figure}
    \centering
    \begin{tikzpicture}[thick,scale=1] 
        \draw[middlearrow={latex reversed}, red] (-2,1) -- (-2,-3);
        \draw[blue] (-1,2) -- (-2,-3);
        \draw[blue] (1,2) -- (-2,-3);
        \draw[blue] (2,-3) -- (-1,2);
        \draw[blue] (2,-3) -- (1,2);
        \draw[middlearrow={latex reversed}, red] (2,-3) -- (2,1);
        \draw[blue] (-2,-3) -- (2,-3);
    
    
        \draw (0,1.5) ellipse (3 and 1.5);
            \draw (3.5,1.5) node [fill=black!0,draw=black!0] {$V_1$};
        \draw (-2,-3) ellipse (1 and 1);
            \draw (-3.5,-3) node [fill=black!0,draw=black!0] {$V_2$};
        \draw (2,-3) ellipse (1 and 1);
            \draw (3.5,-3) node [fill=black!0,draw=black!0] {$V_3$};
        
        \draw (-1,2) node [label=above:{$\scaleto{a_{1}}{8pt}$}] {};
        \draw (1,2) node [label=above:{$\scaleto{a_{2}}{8pt}$}] {};
        \draw (-2,1) node [label=left:{$\scaleto{c_{2}}{8pt}$}] {};
        \draw (2,1) node [label=right:{$\scaleto{b_{1}}{11pt}$}] {};
        \draw (-2,-3) node [label=below:{$\scaleto{b_{2}}{11pt}$}] {};
        \draw (2,-3) node [label=below:{$\scaleto{c_{1}}{8pt}$}] {};

        \begin{scope}[xshift = 9cm]
            \draw[middlearrow={latex reversed}, red] (-2,1) -- (2,1);
            \draw[middlearrow={latex reversed}, red] (-2.5,0) -- (2.5,0);
            \draw[middlearrow={latex reversed}, red] (-2,-1) -- (2,-1);
            
            \draw[blue] (-2,1) -- (2,-1);
            \draw[blue] (-2.5,0) -- (2,1);
            \draw[blue] (-2,-1) -- (2.5,0);
        
        
            \draw (-2,0) ellipse (1.5 and 2);
                \draw (-2,-2.5) node [fill=black!0,draw=black!0] {$V_1$};
            \draw (2,0) ellipse (1.5 and 2);
                \draw (2,-2.5) node [fill=black!0,draw=black!0] {$V_2$};
    
            \draw (-2,1) node [label=left:{$\scaleto{a_{1}}{8pt}$}] {};
            \draw (-2.5,0) node [label=left:{$\scaleto{b_{1}}{11pt}$}] {};
            \draw (-2,-1) node [label=left:{$\scaleto{c_{1}}{8pt}$}] {};
            \draw (2,1) node [label=right:{$\scaleto{c_{2}}{8pt}$}] {};
            \draw (2.5,0) node [label=right:{$\scaleto{a_{2}}{8pt}$}] {};
            \draw (2,-1) node [label=right:{$\scaleto{b_{2}}{11pt}$}] {};
        \end{scope}
    \end{tikzpicture}
    \caption{Copies of $D_2$ inside oriented graphs from $\G$ with index vector $(4,1,1)$ and $(3,3,0)$, respectively. We identify the parts of $D_2$ as $\{a_1, a_2\},\{b_1, b_2\},\{c_1, c_2\}$. The red edges correspond to reversed edges, and blue edges follow the direction from $V_1$ to $V_2$, $V_2$ to $V_3$, and $V_3$ to $V_1$. We omit edges inside parts for a cleaner picture.}
    \label{fig:411pattern}
\end{figure}

We need to split the proof into two cases $s=2$ and $s\ge 3$ since for $s=2$ there are embeddings of $D_2$ into $T(2,k)$ which are unbalanced and not entirely contained in one part (see Claim~\ref{claim:int_pat} and Figure~\ref{fig:411pattern}), which does not occur for $s\ge 3$ (see Claim~\ref{claim:int_pat_dr}). 

\begin{claim} \label{claim:int_pat}
    Let $\varphi: V(D_2) \to V(G)$ be an embedding of $D_2$ into an oriented graph $G$ from the family $\G$ with partition $\CP = (V_1,V_2,V_3)$.
    Then the index vector $\mathbf{i}_\CP(\varphi(D_2))$ is one of the following: $(6,0,0)$, $(0,6,0)$, $(0,0,6)$, $(4,1,1)$, $(1,4,1)$, $(1,1,4)$, $(3,3,0)$, $(3,0,3)$, $(0,3,3)$, or $(2,2,2)$.
\end{claim}

\begin{claimproof} 
    Let $\varphi: V(D_2) \to V(G)$ be an embedding of $D_2$ into an oriented graph $G$ from the family $\G$ with partition $\CP = (V_1,V_2,V_3)$. Since all triples $(u,v,w)$ where $u \ge v \ge w$ and $u+v+w=6$ are 
    $$\{ (6,0,0), (5,1,0), (4,2,0), (4,1,1), (3,3,0), (3,2,1), (2,2,2) \}, $$ 
    by Lemma~\ref{lem:sym}, it is enough to check that the index vector $\mathbf{i}_\CP(\varphi(D_2)) = (|V_1\cap \varphi(D_2)|, |V_2 \cap \varphi(D_2)|, |V_3 \cap \varphi(D_2)|)$ cannot be $(5,1,0)$, $(4,2,0)$, or $(3,2,1)$.

    First, assume $\mathbf{i}_\CP(\varphi(D_2)) = (5,1,0)$.
    Note that two parts of $D_2$ must be embedded into $V_1$. Therefore, by the $K_{1,2}$-freeness of $G_{\CP}$, the vertices from the remaining part of $D_2$ cannot be embedded into $V_2$, contradicting the fact $|V_2 \cap \varphi(D_2)| = 1$.

    Now, assume $\mathbf{i}_\CP(\varphi(D_2)) = (4,2,0)$. Note that at least one part of $D_2$ must be embedded into $V_1$. Then, the two vertices embedded into $V_2$ must come from the following part, but then no vertex from the third part can be embedded into $V_1$, contradicting the fact $|V_1 \cap \varphi(D_2)| > 2$.

    Finally, assume $\mathbf{i}_\CP(\varphi(D_2)) = (3,2,1)$. We argue in a similar way to the last case if there is a part from $D_2$ completely embedded into $V_1$. Otherwise, the three vertices embedded into $V_1$ come from the three different parts of $D_2$. In any such embedding, there are vertices $v_1, v_2, v_3 \in V(D_2)$ from different parts so that $v_1v_2v_3$ is a consistently oriented triangle in $D_2$, $v_3$ is embedded into $V_3$, $v_2$ is embedded into $V_1$, and $v_1$ is embedded into $V_2$. Thus, there is a triangle in $G_{\CP}$, a contradiction. 
\end{claimproof}

\begin{claim} \label{claim:int_pat_dr}
    Let $s\ge 3$ and $\varphi: V(D_s) \to V(G)$ be an embedding of $D_s$ into an oriented graph $G$ from the family $\G$ with partition $\CP = (V_1,V_2,V_3)$.
    Then the index vector $\mathbf{i}_\CP(\varphi(D_s))$ is $(3s,0,0)$, $(0,3s,0)$, $(0,0,3s)$, or $(s,s,s)$.
\end{claim}

\begin{claimproof}
    Let $\varphi: V(D_s) \to V(G)$ be an embedding of $D_s$ into an oriented graph $G$ from the family $\G$ with partition $\CP = (V_1,V_2,V_3)$. By Lemma~\ref{lem:sym}, we can assume the index vector $\mathbf{i}_\CP(\varphi(D_s)) = (u,v,w)$ satisfies $u \ge v \ge w$. 

    We let $A$, $B$, and $C$ denote the three parts of $D_s$, where all edges between parts go from $A$ to $B$, from $B$ to $C$, and from $C$ to $A$. We will split the proof into the following three cases. 

    \medskip \noindent \textbf{Case (1)} There are at least two vertices from the same part of $D_s$ embedded into $V_2$. \medskip
    
    Assume two vertices from $A$ are embedded into $V_2$. Note that no vertex from $B$ can be embedded into $V_1$. As $u \ge s \ge 3$, there are at least two vertices from the same part embedded into $V_1$.

    \medskip \textbf{Subcase (1.1) } Two vertices from $A$ are embedded into $V_1$. \medskip

    Then $B$ has to be completely embedded into $V_2$, and vertices from $C$ cannot be embedded into any part, a contradiction.

    \medskip \textbf{Subcase (1.2) } Two vertices from $C$ are embedded into $V_1$. \medskip

    Then $B$ has to be completely embedded into $V_3$, which implies $w \ge s$, which can only happen when $\mathbf{i}_\CP(\varphi(D_s))=(s,s,s)$.

    \medskip \noindent  \textbf{Case (2) } $w \ge 1$ and $V_2$ contains at most one vertex from each part of $D_s$. \medskip
    
    Since $V_2$ contains at most one vertex from each part of $D_s$, then $v \le 3$. Let us assume, without loss of generality, that there is a vertex from $A$ embedded into $V_3$.

    \medskip \textbf{Subcase (2.1) } There is a vertex from $B$ embedded into $V_2$ and a vertex from $C$ into $V_1$. \medskip
    
    In this case, there is a triangle of reversed orientation, a contradiction.

    \medskip \textbf{Subcase (2.2) } There is a vertex from $B$ embedded into $V_2$ and $C$ is embedded into $V_2 \cup V_3$. \medskip
    
    Then, as there is at most one vertex from $C$ embedded into $V_2$, there is one vertex from $A$ and at least $s-1$ vertices from $C$ embedded into $V_3$.
    We conclude $w \ge s$, which implies $\mathbf{i}_\CP(\varphi(D_s))=(s,s,s)$.
    
    \medskip \textbf{Subcase (2.3) } There is no vertex from $B$ embedded into $V_2$ , but there is a vertex from $A$ embedded into $V_2$. \medskip
 
    Since there is a vertex from $A$ embedded into $V_2$, there could not be two vertices from $B$ embedded into $V_1$. Therefore, there is no vertex from $B$ embedded into $V_2$ and at most one into $V_1$, which implies there are at least $s-1\ge 2$ vertices from $B$ embedded into $V_3$ which implies $2 \ge v \ge w \ge 3$, a contradiction.

    \medskip \textbf{Subcase (2.4) } There is no vertex from $A$ or $B$ embedded into $V_2$. \medskip
    
    As $v\ge w\ge 1$, there is a vertex from $C$ embedded into $V_2$. Then, there could not be two vertices from $A$ embedded into $V_1$. Therefore, there is no vertex from $A$ embedded into $V_2$ and at most one embedded into $V_1$, which implies there are at least $s-1\ge 2$ vertices from $A$ embedded into $V_3$ which implies $1 \ge v \ge w \ge 2$, a contradiction.

    \medskip \noindent  \textbf{Case (3) } $w=0$ and $V_2$ contains at most one vertex from each part of $D_s$. \medskip
    
    Since $s\ge 3$, there are at least two vertices from each part of $D_s$ embedded into $V_1$. Therefore, there could be no vertex embedded into $V_2$, which implies $\mathbf{i}_\CP(\varphi(D_s))=(3s,0,0)$.
\end{claimproof}

\begin{claim} \label{claim:lin_comb}
    The vector $(1,2,3)$ does not belong to the lattice generated by integer linear combinations modulo 6 of the vectors $(2,2,2)$, $(4,1,1)$, $(1,4,1)$, $(1,1,4)$, $(3,3,0)$, $(3,0,3)$, and $(0,3,3)$.
\end{claim}

\begin{claimproof}
    In what follows, all equations are modulo 6. First, note that 
    $(3,3,0) = (2,2,2)+(1,1,4).$ 
    Similarly, $(3,0,3)$, and $(0,3,3)$ can be written as linear integer combinations of $(2,2,2)$, $(4,1,1)$, $(1,4,1)$, and $(1,1,4)$.
    Now, assume that, for some integers $\alpha$, $\beta$, $\gamma$, and $\sigma$, we have that 
    $\alpha (2,2,2)+ \beta (4,1,1) + \gamma (1,4,1) + \sigma (1,1,4) = (1,2,3).$
    Thus, from
    $ 2 \alpha + 4 \beta + \gamma + \sigma = 1$ and
    $ 2 \alpha + \beta + 4 \gamma + \sigma = 2$
    we conclude
    $ 3 (\beta - \gamma) = -1, $
    a contradiction, since $3 (\beta - \gamma) \in \{0,3\}$ mod 6.
\end{claimproof}

\begin{proof}[Proof of Proposition~\ref{prop:dr}]
    Note that $T(s,k)$ is an oriented graph belonging to the family $\G$ with parts $V_1, V_2, V_3$ of size $(|V_1|, |V_2|, |V_3|) = (s-1,s,s+1) + k (s,s,s)$.
    For $s=2$, by Claim~\ref{claim:int_pat}, the index vector modulo 6 of every copy of $D_2$ is $(4,1,1)$, $(1,4,1)$, $(1,1,4)$, $(3,3,0)$, $(3,0,3)$, $(0,3,3)$, $(2,2,2)$, or $(0,0,0)$. Hence, if $T(2,k)$ contains a perfect $D_2$-tiling, there is a linear integer combination of these vectors equal to $(|V_1|, |V_2|, |V_3|)$, contradicting Claim~\ref{claim:lin_comb}.    
    For $s\ge 3$, by Claim~\ref{claim:int_pat_dr}, the index vector modulo $3s$ of every copy of $D_s$ is $(s,s,s)$ or $(0,0,0)$. Therefore, if $T(s,k)$ contains a perfect $D_s$-tiling, then $|V_1|\equiv |V_2| \equiv |V_3|$ modulo $3s$, a contradiction.  
\end{proof}

\section{Minimum semi-degree threshold for $C_6^2$} \label{sec:kappa}

In this section, we will show that $\kappa^0(C_6^2) \ge 3/7$. We will make use of the following claim.

\begin{claim} \label{claim:blowup}
    Let $2\ell+1 \le k \le 4 \ell+1$. If $G$ is an oriented graph with no copy of $C_{k}^\ell$ then a blow-up of $G$ also contains no copy of $C_{k}^\ell$.
\end{claim}

\begin{claimproof}
    Let $V(G)=[n]$.
    Let $G_t$ denote the $t$-blow-up of $G$, that is, there is a partition $( V_1, \dots, V_n )$ of the vertex set $V(G_t)$ so that $|V_i|=t$ for every $i \in [n]$ and we have $ab \in E(G_t)$ if and only if $a \in V_i$ and $b \in V_j$ for some $i,j \in [n]$ for which $ij \in E(G)$.

    The condition $k \le 4 \ell+1$ implies that for every two vertices $u,v \in V(C_k^\ell)$ there is an edge or a 2-path connecting $u$ and $v$, that is, either $uv \in E(C_k^\ell)$ or $vu \in E(C_k^\ell)$, or there is $w \in V(C_k^\ell)$ such that $uw, wv \in E(C_k^\ell)$ or $vw, wu \in E(C_k^\ell)$.

    If $u, v \in V_i$ for some $i \in [n]$, then there is no edge or 2-path connecting $u$ and $v$ in $G_t$. Therefore, for any embedding $\varphi: V(C_k^\ell) \to V(G_t)$ of $C_k^\ell$ into $G_t$, the vertices of $C_k^\ell$ are embedded into different parts, which implies one can obtain an embedding of $C_k^\ell$ into $G$ from $\varphi$.  
\end{claimproof}

The authors of~\cite{DHLMPT} noted that $\kappa^0(C_6^2) \ge 2/5$, since a regular tournament on five vertices does not contain a copy of $C_{6}^2$. However, there exists a regular tournament on seven vertices that does not contain a copy of $C_6^2$.
Let $T$ be the tournament on the vertex set $V(T) \coloneqq \{ v_1, \dots, v_7\}$ with edge set $E(T) \coloneqq \{ v_i v_j : j-i \in \{1,2,4\} \} $, where the indices are taken modulo 7 (see Figure~\ref{fig:tourn_T}). 

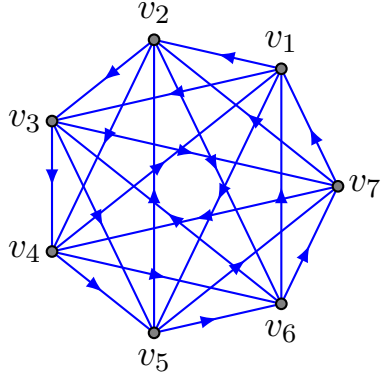
\begin{figure}
    \centering
    \begin{tikzpicture}[thick,scale=1] 
        \foreach \x in {1,...,7} {
        \foreach \y in {1,2,4} {
            \draw[middlearrow={latex}, blue] (360*\x/7 : 2) -- (360*\x/7 + 360*\y/7 : 2);
            }}
    
        \draw (360*1/7 : 2) node [label=above:{$\scaleto{v_{1}}{8pt}$}] {};
        \draw (360*2/7 : 2) node [label=above:{$\scaleto{v_{2}}{8pt}$}] {};
        \draw (360*3/7 : 2) node [label=left:{$\scaleto{v_{3}}{8pt}$}] {};
        \draw (360*4/7 : 2) node [label=left:{$\scaleto{v_{4}}{8pt}$}] {};
        \draw (360*5/7 : 2) node [label=below:{$\scaleto{v_{5}}{8pt}$}] {};
        \draw (360*6/7 : 2) node [label=below:{$\scaleto{v_{6}}{8pt}$}] {};
        \draw (360*7/7 : 2) node [label=right:{$\scaleto{v_{7}}{8pt}$}] {};
    \end{tikzpicture}
    \caption{A regular tournament on 7 vertices with no copy of $C_6^2$.}
    \label{fig:tourn_T}
\end{figure}

\begin{claim} \label{claim:c62free}
    The tournament $T$ defined above does not contain a copy of $C_6^2$.
\end{claim}

\begin{claimproof}
    Assume for the sake of contradiction that $\varphi: V(C_6^2) = \{u_1, \dots, u_6\} \to V(T)$ is an embedding of $C_6^2$ into $T$. Assume without loss of generality that $v_7 \notin \varphi(C_6^2)$ and $\varphi(u_1) = v_6$. Therefore, since $N^+(v_6) \cap \varphi(C_6^2) = \{ v_1, v_3 \}$ and the edge between $v_1$ and $v_3$ goes from $v_1$ to $v_3$, we must have $\varphi(u_2) = v_1$ and $\varphi(u_3) = v_3$.

    As $N^+(v_1) \cap N^+(v_3) = \{v_5\}$, then $\varphi(u_4) = v_5$. Finally, as $N^+(v_3) \cap N^+(v_5) = \{ v_7 \}$, we arrive at a contradiction, since no vertex can be $\varphi(u_5)$.
\end{claimproof}

\begin{proof}[Proof of Proposition~\ref{prop:kappa}]
    By Claims~\ref{claim:blowup} and~\ref{claim:c62free}, the $t$-blow-up $T_t$ of the tournament $T$ has $7t$ vertices, minimum semi-degree $\delta^0(T_t)= 3t$ and does not contain a copy of $C_6^2$. We conclude 
    $$\kappa^0(C_6^2) \ge \lim_{t \to \infty} \frac{3t+1}{7t} = \frac{3}{7}. \eqno\qedhere$$
\end{proof}

\section{The tileability of $D_{1,1,2}$} \label{sec:d112}

In this section, we will use some of the theory and notation developed in~\cite{han,lo-markstrom} to show that the oriented graph $D_{1,1,2}$ is tileable. The proof is based on the stability method, similar to the proof that the consistently oriented triangle is tileable in~\cite{li-molla}. For $\gamma>0$, we say an oriented graph $G$ on $n$ vertices is \emph{$\gamma$-extremal} if there exists a partition $\CP=(V_1, V_2, V_3)$ of $V(G)$ such that, for every $i \in [3]$,
$$(1/3 - \gamma ) n \le |V_i| \le (1/3 + \gamma ) n, $$
and the number of edges oriented from $V_3$ to  $V_2$, from $V_2$ to  $V_1$, and from $V_1$ to $V_3$ are each at most $\gamma n^2$. We call such partition $\CP$, a \emph{$\gamma$-extremal partition} of $V(G)$.

Before we present the proof, we will introduce more notation and definitions to facilitate the argument. The rest of the section is organized as follows. In Section~\ref{subsec:notation}, we introduce the notation used in the proof. In Section~\ref{subsec:prelim}, we prove an almost perfect tiling result. In Section~\ref{subsec:overview}, we introduce auxiliary lemmas that imply Theorem~\ref{thm:d112}. We defer the proof of the auxiliary lemmas to Sections~\ref{subsec:non_ext} and~\ref{subsec:ext}.

\subsection{Notation}\label{subsec:notation}

Recall that $C_3$ is the consistently oriented triangle, and for simplicity, we denote $D=D_{1,1,2}$.
Given oriented graphs $F,G$, let $H = H_{F}$ be the $|F|$-uniform hypergraph with $V(H)=V(G)$, and $e \in E(H)$ if and only if $G[e]$ induces a copy of $F$.
We denote $H_3 \coloneqq H_{C_3}$ and $H_4 \coloneqq H_{D}$.
For a hypergraph $H$, let $\delta_1(H)$ denote the minimum degree of $H$.

For sets $A,B\subseteq V(G)$, define 
$$E^\sigma(A,B)=\bigcup_{v\in A}N^\sigma(v,B)\qquad \text{and} \qquad e^\sigma(A,B)=|E^\sigma(A,B)|$$ where $\sigma\in\{+,-\}$.
Subsequently, define $E(A,B)=E^+(A,B)\cup E^-(A,B)$ and $e(A,B)=|E(A,B)|$.
Given some $\veps>0$, we say that the ordered pair $(A,B)$ is \textit{$\veps$-regular} if for every $X\subseteq A$ and $Y\subseteq B$ with $|X|\ge \veps|A|$ and $|Y|\ge \veps |B|$, we have that 
$$\left|\frac{e^+(X,Y)}{|X||Y|}-\frac{e^+(A,B)}{|A||B|}\right|\le \veps.$$ 

For an edge $uv\in E(G)$ and a subset $A\subseteq V(G)$, define 
$$d^{\sigma,\tau}(uv,A)=|N^\sigma(u,A)\cap N^\tau(v,A)|,$$ where $\sigma,\tau\in\{+,-\}$.
For a set $A\subseteq V(G)$ and some positive constant $\beta<1$, define $$N^\sigma_\beta(A)=\{v\in V(G) \setminus A :d^\tau(v,A)\ge |A|-\beta n\},$$ where $\{\sigma,\tau\}=\{+,-\}$.
For a set $U\subseteq V(G)$, let $d(U)$ be the number of copies of $D$ contained in $G[U]$.
For sets $W,X,Y,Z\subseteq V(G)$, let $d(W,X,Y,Z)$ be the number of copies of $D$ with vertex set $\{w,x,y,z\}$ such that $w\in W$, $x \in X$, $y \in Y$, and $z\in Z$.

Given a $k$-uniform hypergraph $H$, vertices $x,y\in V(H)$, and $\beta>0$, we say that a set $S\subseteq V$ is an $(H,x,y)$-linking $(k\ell-1)$-set if both $H[S\cup\{x\}]$ and $H[S\cup\{y\}]$ have perfect matchings and $|S|=k\ell-1$.
The vertices $x, y \in V(H)$ are $(H,\beta,\ell)$-reachable if there are at least $\beta n^{k\ell-1}$ $(H,x,y)$-linking $(k\ell-1)$-sets.
A set $U\subseteq V$ is $(H,\beta,\ell)$-closed if every pair of distinct vertices in $U$ is $(H,\beta,\ell)$-reachable.
A partition $\CP$ of $V(H)$ is $(H,\beta,\ell)$-closed if every set in $\CP$ is $(H,\beta,\ell)$-closed.
For every $x\in V(H)$, let $\wt N_H(\beta,\ell,x)$ be the set of vertices $y$ such that $x,y$ are $(H,\beta,\ell)$-reachable.

Recall that, given a partition $\CP=(V_1,\dots,V_d)$ of $V(H)$ and a subset $U\subseteq V(H)$, the \emph{index vector} of $U$ with respect to $\CP$, denoted by $\mathbf{i}_\CP(U)$, is defined as $\mathbf{i}_\CP(U)=(|U\cap V_1|,\dots,|U\cap V_d|)$.
Let $I_\CP(H)=\{\mathbf{i}_\CP(e):e\in E(H)\}$ be the set of \emph{edge-vectors} of $H$ with respect to $\CP$.
For $\mu>0$, let
$$I_\CP^\mu(H)=\{\mathbf{v}\in I_\CP(H):\text{there are at least $\mu n^k$ edges $e$ in $H$ with $\mathbf{i}_\CP(e)=\mathbf{v}$}  \}$$ be the set of \emph{$\mu$-robust edge-vectors}.

We may refer to an additive subgroup of $\mathbb{Z}^d$ as a \emph{lattice}.
Let $L_\CP(H)$, $L_\CP^\mu(H)$ be the lattices generated by $I_\CP(H)$, $I_\CP^\mu(H)$ respectively.
For $i\in [d]$, let $\mathbf{e}_i$ be the \emph{$i$-th unit vector} in $\mathbb{Z}^d$.
A \emph{transferral} is a vector $\mathbf{v}\in\mathbb{Z}^d$ such that there exist distinct $i,j\in[d]$ with $\mathbf{v}=\mathbf{e}_i-\mathbf{e}_j$.
A \emph{2-transferral} $\mathbf{v}\in L_\CP^\mu(H)$ is a transferral such that $\mathbf{v}=\mathbf{v}_1-\mathbf{v}_2$ for some $\mathbf{v}_1,\mathbf{v}_2\in I_\CP^\mu(H)$.
The lattice $L_\CP^\mu(H)$ is \emph{$2$-transferral-free} if it contains no $2$-transferral.

\subsection{Almost perfect tiling via regularity} \label{subsec:prelim}

In this section, we will obtain an almost perfect $D$-tiling result (namely, Theorem~\ref{thm:almost_perfect} below) from the celebrated regularity and blow-up lemmas. More specifically, we use the following degree form of the Regularity Lemma for digraphs. The digraph version of the regularity lemma can be proven in a similar way to the undirected version~\cite{alon}, and its degree form can be obtained exactly as in the undirected case. See~\cite{regularity} for a survey on the regularity lemma.

\begin{lemma}[Degree form of the Diregularity Lemma]\label{lem:di-regularity}
For every $\veps>0$ and integer $M$, there are some integers $M'$ and $n_0$ such that for every directed graph $G$ with $n\ge n_0$ vertices and $d \in [0,1]$, there is a partition of $V(G)$ into sets $V_0,V_1,\dots,V_\ell$ and a spanning subgraph $G'$ of $G$ such that the following holds:
\begin{itemize}
    \item $M\le \ell \le M'$,
    \item $|V_0|\le \veps n$,
    \item $|V_1|=|V_2|=\dots=|V_\ell|$,
    \item for $\sigma\in\{+,-\}$, $d^\sigma_{G'}(v)\ge d^\sigma_G(v)-(d+\veps)n$,
    \item for $1\le i\le \ell$, $G'[V_i]$ is empty,
    \item for every $1\le i,j\le \ell$ with $i \neq j$, the pair $(V_i,V_j)$ is $\veps$-regular in $G'$ and has density either $0$ or at least $d$.
\end{itemize}       
\end{lemma}

For an oriented graph $G'$ obtained from Lemma~\ref{lem:di-regularity}, with vertex partition $V(G') = V_1 \cup \dots \cup V_\ell$, we define the $(\veps, d)$-reduced digraph $R'$ on vertex set $[\ell]$ with $ij \in E(R')$ if and only if $(V_i, V_j)$ is an $\veps$-regular pair of density at least $d$ (in other words, $ij$ is an edge in $R'$ if and only if there is an edge from $V_i$ to $V_j$ in $G'$). The reduced graph $R'$ may not be an oriented graph even if $G$ is. However, the next lemma from~\cite{kelly} allows us to obtain a spanning oriented subgraph from $R'$ preserving the minimum semi-degree from $G$.

\begin{lemma}\label{lem:reduced_graph}
    For every $\veps \in (0,1)$ there exist numbers $M = M(\veps)$ and $n_0 = n_0(\veps) $ such that the following holds. Let $d \in [0, 1]$ and let $G$ be an oriented graph of order $n \ge n_0$ and let $R'$ be the reduced digraph with parameters $(\veps, d)$ obtained by applying the Lemma~\ref{lem:di-regularity} to $G$ with parameters $\veps$, $d$ and $M$. Then $R'$ has a spanning oriented subgraph $R$ such that $\delta^0(R) \ge (\delta^0(G)/|G|-(d+3\veps))|V(R)|$.
\end{lemma}

The following result implies that there is an almost perfect triangle tiling in the reduced graph.

\begin{theorem}[Keevash, Sudakov~\cite{keevash}] \label{thm:keevash}
    There are $c' > 0$ and $n_0 \in \N$ so that every oriented graph $G$ on $n\ge n_0$ vertices with minimum semi-degree $\delta^0(G) \ge (1/2 - c')n$ contains a consistently oriented triangle tiling covering all but at most 3 vertices.
\end{theorem}

Next, we need the following result to obtain an almost perfect $D$-tiling in $G$ from an almost perfect $C_3$-tiling in $R$.

\begin{lemma}[Blow-up Lemma~\cite{blowuplemma}]\label{lem:blowup}
    Given a graph $F$ on $[s]$ and positive numbers $d$, $\Delta$, there is a positive real $\eta_0 = \eta_0(d, \Delta, s)$ such that the following holds for every positive number $t$ and every $0 < \eta \le \eta_0$. Let $F'$ be the graph obtained from $F$ by replacing each vertex $i \in V(F)$ with a set $V_i$ of $t$ new vertices and joining all vertices in $V_i$ to all vertices in $V_j$ whenever $ij$ is an edge of $F$. Let $G'$ be a spanning subgraph of $F'$ such that for every edge $ij \in E(F)$ the bipartite graph $G'[V_i,V_j]$ is $\eta$-regular and has minimum degree at least $dt$. Then $G'$ contains a copy of every subgraph $H$ of $F'$ with $\Delta(H) \le \Delta$. 
\end{lemma}

Our goal is to use Lemma~\ref{lem:blowup} to obtain a $D$-tiling in $G$ that covers each copy of $C_3$ in $R$.
For each copy $ijk$ of $C_3$ in $R$, we let $H=H(i,j,k)$ be the graph on vertex set $V_i \cup V_j \cup V_k$ with (unordered) edges being the underlying edges from $E^+_{G'}(V_i, V_j) \cup E^+_{G'}(V_j, V_k) \cup E^+_{G'}(V_k, V_i)$.
If the host oriented graph is a tournament, then this can be done directly: any copy of $K_{1,1,2}$ in the auxiliary graph $H$ guarantees a copy of $D$ in the tournament. However, this is no longer the case when the host oriented graph is not a tournament, even if it is a near-tournament. We circumvent this issue via the following simple lemma, which guarantees a $D$-tiling covering a copy of $K_{4,4,4}$ in the auxiliary graph of a near-tournament.

\begin{lemma} \label{lem:k444}
    Let $G$ be a near-tournament, and $V_1, V_2, V_3 \subseteq V(G)$ be disjoint sets, each of size four. We let $H$ be the graph on vertex set $V_1 \cup V_2 \cup V_3$ with (unordered) edges being the underlying edges from $E^+_{G}(V_1, V_2) \cup E^+_{G}(V_2, V_3) \cup E^+_{G}(V_3, V_1)$. If $H$ is isomorphic to $K_{4,4,4}$, then $G[V_1 \cup V_2 \cup V_3]$ contains a perfect $D$-tiling.
\end{lemma}

\begin{proof}
    Since $G$ is a near-tournament, each set $V_i$ contains at least one directed edge in $G$. Let $c_1c_2$, $c_3c_4$, and $c_5c_6$ be such edges inside parts $V_1, V_2$, and $V_3$, respectively. Label the remaining vertices of $H$ so that 
    $$V_1 = \{a_3, b_2, c_1, c_2 \}, \quad
    V_2 = \{a_1, b_3, c_3, c_4 \}, \quad
    V_3 = \{a_2, b_1, c_5, c_6 \}.$$

    The sets $\{a_1, b_1, c_1, c_2\}$, $\{a_2, b_2, c_3, c_4\}$, and $\{a_3, b_3, c_4, c_5\}$ induce vertex-disjoint copies of $D$.
\end{proof}

\begin{theorem}\label{thm:almost_perfect}
    There exist $c > 0$ and $n_0 \in \N$ such that, for every $\veps \in (0,1)$, every near-tournament~$G$ on $n\ge n_0$ vertices with minimum semi-degree $\delta^0(G) \ge (1/2 - c)n$ contains a $D$-tiling covering all but at most $\veps n$ vertices.
\end{theorem}

\begin{proof}
    Let $c'>0$ and $n_0$ be given by Theorem~\ref{thm:keevash}. We set $c = c'/2$ and $d>0$ sufficiently small so that for $\veps < \eta_0(d, 12, 3)$ given by Lemma~\ref{lem:blowup} we have $c+d+\veps < c'$. Finally, let $M > \max\{5/\veps, n_0\}$.
    By moving at most eleven vertices from each $V_i$ to $V_0$ we can assume that $|V_1|=\dots=|V_\ell|$ from Lemma~\ref{lem:di-regularity} is a multiple of 12.
    We apply Lemma~\ref{lem:di-regularity} with parameters $\veps/3$ and $M$ to obtain $G'$ and the $(\veps/3, d)$-reduced graph $R'$ from $G'$.
    By Lemma~\ref{lem:reduced_graph}, there exists an oriented subgraph $R$ of $R'$ with $\delta^0(R)\ge (1/2-c-d-\veps) |V(R)|$.
    By Theorem~\ref{thm:keevash}, there exists a $C_3$-tiling in $R$ covering all but at most three vertices of $V(R)$. 

    For each copy $ijk$ of $C_3$ in $R$, we let $H=H(i,j,k)$ be the graph on vertex set $V_i \cup V_j \cup V_k$ with (unordered) edges being the underlying edges from $E^+_{G'}(V_i, V_j) \cup E^+_{G'}(V_j, V_k) \cup E^+_{G'}(V_k, V_i)$. It follows from Lemma~\ref{lem:blowup} that there exists a perfect $K_{4,4,4}$-tiling in $H$. By Lemma~\ref{lem:k444}, since $G$ is a near-tournament, this corresponds to a $D$-tiling in $G$ covering all vertices except the vertices from at most three parts and the exceptional part $V_0$. Therefore, this $D$-tiling misses at most
    $$ 3\cdot \frac{n}{\ell} + \frac{\veps n}{3} < \veps n$$
    vertices, where the last inequality follows the fact $\ell \ge M > 5/\veps$.
\end{proof}

\begin{remark}
    We highlight that the general framework used in this section can be used to obtain an almost perfect $D_{a,b,c}$-tiling for every $a\le b \le c$, not only for $D_{1,1,2}$.
\end{remark}

\subsection{Proof of Theorem~\ref{thm:d112}} \label{subsec:overview}

We let $G$ be a semi-regular near-tournament on $n$ vertices where $n$ is a multiple of 4. Our goal is to show that there is a perfect $D$-tiling in $G$. We break the proof into two cases depending on whether $G$ is $\gamma$-extremal or not. 
First, we use the so-called absorbing method to obtain a perfect $D$-tiling when $G$ is not $\gamma$-extremal.

\begin{lemma}[Absorbing Lemma - Lemma 1.1 in~\cite{lo-markstrom}] \label{lem:absorbing}
    Given $k, \ell \in \N$, and $\beta > 0$, there exists $n_0 \in \N$ so that the following holds. If $H$ is a $k$-uniform hypergraph on $n\ge n_0$ vertices and $V(H)$ is $(H, \beta, \ell)$-closed, then there exists a subset $U \subseteq V(H)$ of size $|U| \le \frac{(\beta/2)^{k}n}{4\ell k(k-1)} $ with $|U|$ divisible by $k$ such that, for every $W \subseteq V(H) \setminus U$ with $|W|\le \frac{(\beta/2)^{2k}n}{32\ell^2 k(k-1)^2}$ and $|W|$ divisible by $k$, there exists a perfect matching in $H[U \cup W]$.  
\end{lemma}

\begin{lemma} \label{lem:non_ext}
    Given $\ell \in \N$, there exist $\beta>0$ and $n_0\in \N$ so that the following holds. Suppose $G$ is a semi-regular near-tournament on $n\ge n_0$ vertices, for $n$ divisible by 4, and $V(G)$ is $(H_4(G),\beta,\ell)$-closed. Then, $G$ has a perfect $D$-tiling.
\end{lemma}

\begin{proof}
    Let $c>0$ be given by Theorem~\ref{thm:almost_perfect}, and $\beta, \veps>0$ be such that $\beta^4 < 2^{8} \cdot 3 \ell  c \,$ and $2^{15} 3^2 \ell^2 \veps < \beta^{8}$. We assume $n_0$ is sufficiently large for Theorem~\ref{thm:almost_perfect} and Lemma~\ref{lem:absorbing} to hold for $n\ge n_0$.
    Let $U \subseteq V(G)$ be given by Lemma~\ref{lem:absorbing} applied to $H_4(G)$. The near-tournament induced by $V' \coloneqq V(G)\setminus U$ has $n' \coloneqq n-|U|$ vertices and minimum semi-degree at least 
    $$ \left\lfloor \frac{n-1}{2} \right\rfloor - |U| \ge
    \left( \frac{1}{2} - \frac{|U|}{n} \right) \cdot n' \ge
    \left( \frac{1}{2} - \frac{\beta^4}{2^{8} \cdot 3 \ell } \right) \cdot n' \ge
    \left( \frac{1}{2} - c \right) \cdot n'. $$
    By Theorem~\ref{thm:almost_perfect}, $G[V']$ contains a $D$-tiling $\D _1$ covering all but a set $W$ of at most 
    $$|W| \le \veps n < \frac{\beta^{8}n}{2^{15} 3^2 \ell^2}$$
    vertices. By the property of $U$ given by Lemma~\ref{lem:absorbing}, there is a perfect $D$-tiling $\D_2$ in $G[U \cup W]$. Therefore, $\D_1 \cup \D_2$ is a perfect perfect $D$-tiling in $G$.
\end{proof}

Next, we note that whenever $G$ does not satisfy the conditions of Lemma~\ref{lem:non_ext}, then $G$ is $\gamma$-extremal.

\begin{lemma}\label{lem:notclosedtoextremal}
    Given $\gamma >0$, there exist $c, \beta>0$ such that the following holds. 
    Suppose $G$ is an oriented graph on $n$ vertices.
    If $\delta^0(G)\ge(1/2-c)n$ and for every positive integer $\ell \le 5^{4}\cdot (2^{4}+1)$, $V(G)$ is not $(H_4(G),\beta,\ell)$-closed, then $G$ is $\gamma$-extremal.
\end{lemma}

We prove Lemma~\ref{lem:notclosedtoextremal} in Section~\ref{subsec:non_ext}.
Finally, Theorem~\ref{thm:d112} will follow from the following lemma, which will be proved in Section~\ref{subsec:ext}.

\begin{lemma}[Extremal case] \label{lem:extremal}
    There exist $c, \gamma >0$ and $n_0\in \N$ so that the following holds. Suppose that $n\ge n_0$ is divisible by $4$ and $G$ is a near-tournament on $n$ vertices. 
    If $\delta^0(G)\ge (1/2-c)n$ and $G$ is $\gamma$-extremal, then $G$ has a perfect $D$-tiling.
\end{lemma}

\subsection{Non-extremal case}\label{subsec:non_ext}

The proof of Lemma~\ref{lem:notclosedtoextremal} is based on the following lemmas.
Recall that a \emph{2-transferral} $\mathbf{v}\in L_\CP^\mu(H)$ is a transferral such that $\mathbf{v}=\mathbf{v}_1-\mathbf{v}_2$ for some $\mathbf{v}_1,\mathbf{v}_2\in I_\CP^\mu(H)$.
The lattice $L_\CP^\mu(H)$ is \emph{$2$-transferral-free} if it contains no $2$-transferral.

\begin{lemma}[Lemma 3.8 in~\cite{han}]\label{lem:partition}
    Given $\delta,\delta'>0$, there exist constants $\alpha, \eta>0$ for which the following holds.
    Let $H$ be a $k$-uniform hypergraph on $n$ vertices satisfying $|\wt N_H(\alpha, 1, v)|\ge \delta' n$ for every $v\in V(H)$ and $\delta_1(H)\ge \delta\binom{n-1}{k-1}$.
    Then there exists a $(H,\eta,2^{\lfloor 1/\delta\rfloor-1})$-closed partition $\CP=\{V_1,\dots,V_d\}$ for some $d\le \min\{\lfloor 1/\delta\rfloor,\lfloor 1/\delta'\rfloor\}$ such that $|V_i|\ge (\delta'-\alpha)n$ for every $i\in[d]$.
\end{lemma}

\begin{lemma} [Lemmas 13 and 14 from~\cite{li-molla}]\label{lem:free2extremal}
    Given $\gamma, \veps >0$, there exist $c, \mu_3>0$ such that the following holds. Suppose $G$ is an oriented graph on $n$ vertices such that $\delta^0(G)\ge(1/2-c)n$.
    If $\CP$ is a non-trivial $\veps$-partition of $V(G)$, $L_\CP^{\mu_3}(H_3(G))$ is $2$-transferral free, then $G$ is $\gamma$-extremal.
\end{lemma}

In order to apply Lemma~\ref{lem:partition} to the hypergraph $H_4(G)$, we first need to obtain lower bounds for $\delta_1(H_4)$ and $|\wt N_{H_4}(\alpha, 1, v)|$, which is the content of the following lemmas.

\begin{lemma}[Lemma 2.1 in~\cite{keevash}]\label{lem:keevash}
    Suppose that $c>0$ and $G$ is an oriented graph on $n$ vertices with $\delta^0(G)\ge (1/2-c)n$. Then, for each vertex $v\in V(G)$, we have that $e^+ (N^+(v),N^-(v))$ is at least $(1/8-2c)n^2$ and at most $(1/8+2c)n^2$.
\end{lemma}

\begin{lemma}\label{lem:minimum_degree}
    Suppose that $c>0$ and $G$ is an oriented graph on $n$ vertices with $\delta^0(G)\ge (1/2-c)n$, then each vertex of $G$ belongs to at least $(1/32-4c)n^3$ copies of $D$.
\end{lemma}

\begin{proof}
    Let $v\in V(G)$. 
    We give a lower bound on the number of subsets $A\subseteq V(G)$ such that $v\in A$, $G[A]$ induces a copy of $D$ and $d_A^+(v)=1,d_A^-(v)=2$.
    For each $u\in N^+(v)$, if we pick two vertices $w,x$ from $|N^+(u)\cap N^-(v)|$, then both $vuw,vux$ are cyclic triangles.
    By Lemma~\ref{lem:keevash}, 
    $$\left( \frac{1}{8} - 2c \right) n^2 \le 
    \sum_{u\in N^+(v)}|N^+(u)\cap N^-(v)|=e^+(N^+(v),N^-(v))
    \le \left( \frac{1}{8} + 2c \right) n^2 .$$
    For every $w\in N^+(u)\cap N^-(v)$, there are at most $2cn$ vertices $x\in N^+(u)\cap N^-(v)$ such that $wx,xw\not\in E(G)$.
    Then we have that there are least
    \begin{align*}
        & \sum_{u\in N^+(v)}\left(\binom{|N^+(u)\cap N^-(v)|}{2}-\frac{2cn^2}{2}\right) \\
        = \,
        & \frac{1}{2}\sum_{u\in N^+(v)}|N^+(u)\cap N^-(v)|^2-\frac{1}{2}\sum_{u\in N^+(v)}|N^+(u)\cap N^-(v)|- \left( \frac{1}{2} +c \right) cn^3 \\
        \ge \,
        & \frac{e^+(N^+(v),N^-(v))^2}{2\delta^+(v)}-\frac{1}{2}e^+(N^+(v),N^-(v))-\left( \frac{1}{2} +c \right) cn^3 \\ 
        \ge \,
        & \frac{1}{2} \cdot \frac{(1/8 - 2c)^2}{(1/2 + c)} \cdot n^3- \frac{1}{2} \left( \frac{1}{8} + 2c \right) n^2 - \left( \frac{1}{2} +c \right) cn^3 \ge \left( \frac{1}{64} - 2c \right) n^3 
    \end{align*}
    many subsets $A$. Similarly, there are at least $\left( \frac{1}{64} - 2c \right) n^3$ many subsets $A\subseteq V(G)$ such that $v\in A$, $G[A]$ induces a copy of $D$ and $d_A^+(v)=2,d_A^-(v)=1$.
\end{proof}

\begin{lemma}[Lemma 19 from~\cite{li-molla}]\label{lem:h3_reachable}
    For every $\beta > 0$, there exists $c>0$ such that the following holds. If $G$ is an oriented graph with $\delta^0(G)\ge (1/2-c)n$, 
    then, for every $v\in V$, $|\wt N_{H_3}(\beta,1,v)|\ge (1/8-10\beta)n$.
\end{lemma}

\begin{lemma}\label{lem:h4_reachable}
    Let $\beta, c > 0$ satisfy Lemma~\ref{lem:h3_reachable}, and $ \alpha < \left( 1/8-10\beta-2c \right) \beta^2 $. Then the following holds. 
    If $G$ is an oriented graph with $\delta^0(G)\ge (1/2-c)n$, 
    then for every $v\in V(G)$, $|\wt N_{H_4}(\alpha,1,v)|\ge \left( 1/8-10\beta-2c \right) \beta n$.
\end{lemma}

\begin{proof}
    Fix $v\in V(G)$ and consider an auxiliary bipartite graph $G_v=(A,B)$ where $B=\binom{V(G)\setminus\{v\}}{2}$ and $A\subseteq\wt N_{H_3}(\beta,1,v)\cap (N^+(v)\cup N^-(v))$.
    For an ordered tuple $(x,y,z)$, there is an edge between $z\in A$ and $\{x,y\}\in B$ only if $xy\in E(G)$ and both $vxy,zxy$ are cyclic triangles in $G$ (in other words, $z$ is adjacent to $\{x, y\}$ in $G_v$ only when
    $\{ x, y\}$ is an $(H_3, z, v)$-linking set).
    From Lemma~\ref{lem:h3_reachable} and the definition of $\wt N_{H_3}(\beta,1,v)$ it follows that one can assume that $|A|=(1/8-10\beta-2c)n$ and that the degree of every vertex $u \in A$ is $d_{G_v}(u) = \beta n^2$.
    
    For each $2$-walk in $G_v$ that starts from some $u\in A$ and ends at some $w\in A\setminus \{u\}$, by definition there is some $\{x,y\}\in B$ such that both $G[\{u,v,x,y\}]$ and $G[\{w,v,x,y\}]$ contain a copy of $D$, that is, $\{x,y,w\}$ is a $(H_4,u,v)$-linking set.
    It suffices to show that there are at least $\beta |A|$ many vertices in $A$ that are contained in at least $\alpha n^3$ many $2$-walks starting and ending in $A$.

    First, note that $|E(G_v)| = (1/8-10\beta-2c)\beta n^3$.
    The number of $2$-walks starting from $A$ is
     \begin{align*}
        \sum_{u\in A}\sum_{\{x,y\}\in N(u)} d_{G_v}(\{x,y\})=&\sum_{\{x,y\}\in B}d_{G_v}(\{x,y\})^2\\\ge& \frac{(\sum_{\{x,y\}\in B}d_{G_v}(\{x,y\}))^2}{|B|}\\\ge& \frac{|E(G_v)|^2}{n^2/2}\ge 2(1/8-10\beta-2c)^2\beta^2 n^4.
    \end{align*}
    On average, each vertex in $A$ is in at least $2(1/8-10\beta-2c)\beta^2 n^3$ many $2$-walks and each vertex in $A$ is in at most $(1/8-10\beta-2c)\beta n^3$ such $2$-walks.
    Then there are at least $\beta|A|$ vertices contained in at least $\alpha n^3$ many $2$-walks, since
    $$ \beta |A| \cdot \left( 1/8-10\beta-2c \right) \beta n^3 + (1-\beta |A|) \alpha n^3 < 2 \left( 1/8-10\beta-2c \right) \beta^2 n^3 . \eqno\qedhere$$
\end{proof}

Now we can apply Lemma~\ref{lem:partition} to $H_4(G)$ and obtain the following.

\begin{lemma}\label{lem:closed_part}
    There exist $c, \eta >0$ and $n_0\in \N$ so that the following holds. Suppose that $G$ is an oriented graph on $n \ge n_0$ vertices such that $\delta^0(G)\ge (1/2-c)n$.
    Then there exists $\CP=\{V_1,\dots,V_d\}$ a $(H_4(G),\eta,2^{4})$-closed partition of $V(G)$ such that $d\le 5$ and $|V_i|\geq 10^{-5} n$ for every $i\in [d]$.
\end{lemma}

\begin{proof}
    Throughout the proof, we assume that $c>0$ is sufficiently small and $n_0$ is sufficiently large so that all inequalities hold true. Let $\delta'=10^{-4}$, $\delta=6/35$, and $\alpha, \eta > 0$ be given by Lemma~\ref{lem:partition}. We also set $\beta = 10^{-2}$ and assume $\alpha < \left( 1/8-10\beta-2c \right) \beta^2$.
    
    By Lemma~\ref{lem:minimum_degree}, we have that 
    $$\delta_1(H_4(G)) \ge 
    \left( \frac{1}{32} - 4c \right) n^3 > 
    \delta \binom{n-1}{3}.$$
    By Lemma~\ref{lem:h4_reachable}, we have that 
    $$|\wt N_{H_4}(\alpha,1,v)|\ge \delta' n$$
    for every $v\in V(G)$.
    Then, by Lemma~\ref{lem:partition}, there exists a $(H_4(G),\eta,2^{4})$-closed partition $\CP= \{ V_1,\dots,V_d \}$ such that $d\le 5$ and $|V_i|\ge (\delta'-\alpha)n\ge 10^{-5} n$.
\end{proof}

In order to obtain a non-trivial $\veps$-partition $\CP$ for which $L_\CP^{\mu_3}(H_3(G))$ is $2$-transferral free, so that we can apply Lemma~\ref{lem:free2extremal}, we will need the following lemmas.

\begin{lemma}\label{lem:merge}
    Suppose $\ell, d \in \N$, $0< \beta' \le 1/d$, and $\mu_4 > 0$. Then there are $\beta>0$ and $n_0 \in \N$ such that the following holds. If $G$ is an oriented graph on $n\ge n_0$ vertices and $\CP = \{V_1, \dots, V_d\}$ is a $(H_4,\beta',\ell)$-closed partition such that there exists a $2$-transferral $\mathbf{u_i}-\mathbf{u_j}\in L_{\CP}^{\mu_4}(H_4)$, then merging $V_i,V_j$ gives a $(H_4,\beta,5\ell+1)$-closed partition.
\end{lemma}

Lemma~\ref{lem:merge} is analogous to Lemma 11 in~\cite{li-molla}, and its proof is almost identical to the one presented in~\cite{li-molla}. For the sake of completeness, we present the proof of Lemma~\ref{lem:merge} in Appendix~\ref{appendix}.

\begin{lemma}\label{lem:free2free}
    Let $G$ be an oriented graph on $n$ vertices such that $\delta^0(G)\ge (1/2-c)n$ and assume that $\mu_4 < \mu_3^2/3 - c$.
    If $\CP$ is a non-trivial partition of $V(G)$ and $L_\CP^{\mu_4}(H_4(G))$ is $2$-transferral free, then $L_\CP^{\mu_3}(H_3(G))$ is also $2$-transferral free.
\end{lemma}

\begin{proof}
We show the following stronger relation between $I_\CP^{\mu_3}(H_3(G))$ and $I_\CP^{\mu_4}(H_4(G))$.

\begin{claim*}
    Let $\CP = (V_1, \dots, V_d)$ be a non-trivial partition of $V(G)$. Assume that $\mu_4 < \mu_3^2/3 - c$ and $\mathbf{e_i}+\mathbf{e_j}+\mathbf{e_k} \in I_\CP^{\mu_3}(H_3(G))$ for some (not necessarily distinct) $i,j,k \in [d]$, then $2\mathbf{e_i}+\mathbf{e_j}+\mathbf{e_k} \in I_\CP^{\mu_4}(H_4(G))$.
\end{claim*}

\begin{claimproof}
    By the definition of $I_\CP^{\mu_3}(H_3(G))$, there are at least $\mu_3 n^3$ consistently oriented triangles with vertices in $V_i$, $V_j$, and $V_k$, that is,
    $$\sum_{uv\in E(V_j,V_k)}d^{-,+}(uv,V_i)\ge \mu_3 n^3.$$
    Hence, the number of copies of $D$ that have one vertex in each of $V_j,V_k$ and two vertices in $V_i$ is at least
    \begin{align*}
        \sum_{uv\in E(V_j,V_k)}\binom{d^{-,+}(uv,V_i)}{2}-cn^2 \ge \frac{\left(\sum_{uv\in E(V_j,V_k)}d^{-,+}(uv,V_i)\right)^2}{3|E(V_j,V_k)|}-cn^4 \ge \left( \frac{\mu_3^2}{3} - c \right) n^4>\mu_4 n^4,
    \end{align*}
    where, for each edge $uv\in E(V_j,V_k)$, we need to subtract the pairs $\{w, z\}\subseteq V_i$ such that $wz,zw\not\in E(G)$, and there are at most $cn\cdot |V_i| \le cn^2$ of these pairs in total.
\end{claimproof}
    
    By the Claim, if there exists $i \neq j$ such that $2\mathbf{e_i}+\mathbf{e_j} \in I_\CP^{\mu_3}(H_3(G))$, then $$3\mathbf{e_i}+\mathbf{e_j}, 2\mathbf{e_i}+2\mathbf{e_j} \in I_\CP^{\mu_4}(H_4(G)),$$
    and $\mathbf{e_i}-\mathbf{e_j} \in L_\CP^{\mu_4}(H_4(G))$ is a 2-transferral.
    Similarly, if there are distinct $i,j,k$ such that $\mathbf{e_i}+\mathbf{e_j}+\mathbf{e_k} \in I_\CP^{\mu_3}(H_3(G))$, then $L_\CP^{\mu_4}(H_4(G))$ has a 2-transferral.   
    Since every pair of vectors $\textbf{v}_1,\textbf{v}_2\in I_\CP^{\mu_3}(H_3(G))$ that witness a $2$-transferral in $L_\CP^{\mu_3}(H_3(G))$ contains a vector as above, we conclude the lemma. 
\end{proof}

We are now ready to prove Lemma~\ref{lem:notclosedtoextremal}.

\begin{proof}[Proof of Lemma~\ref{lem:notclosedtoextremal}]
    Let $\mu_3 >0$ be given by Lemma~\ref{lem:free2extremal} applied with $\veps = 10^{-5}$. Also, let $\mu_4>0$ be so that we can apply Lemma~\ref{lem:free2free}. 
    By Lemma~\ref{lem:closed_part}, there exists a $(H_4(G),\eta,2^{4})$-closed partition $\CP'=\{V_1,\dots,V_d\}$ of $V(G)$ such that $d\le 5$ and $|V_i|\geq 10^{-5} n$ for every $i\in [d]$.
    If $L_{\CP'}^{\mu_4}(H_4(G))$ contains a $2$-transferral $\textbf{u}_i-\textbf{u}_j$ for distinct $i,j\in [d]$,
    then we consider the new partition $\CP'-V_i-V_j+(V_i\cup V_j)$.
    We may continue to merge in this way until we get a partition $\CP$ such that $L_{\CP}^{\mu_4}(H_4(G))$ is $2$-transferral free. Thus, by Lemma~\ref{lem:free2free}, $L_{\CP}^{\mu_3}(H_3(G))$ is also $2$-transferral free.
    By Lemma~\ref{lem:merge}, we may assume that $\CP$ is a $(H_4(G),\beta,\ell)$-closed partition of $V(G)$ for some $\beta>0$ and $\ell\le 5^{4}\cdot (2^{4}+1)$.   
    If $|\CP|=1$, then $V(G)$ is $(H_4(G),\beta,\ell)$-closed which is a contradiction, so $\CP$ is non-trivial.
    Then, by Lemma~\ref{lem:free2extremal}, $G$ is $\gamma$-extremal.
\end{proof}

\subsection{Extremal case}\label{subsec:ext}

For a partition $\CP =\{ V_1,V_2,V_3 \}$ of $V(G)$, we may abuse the notation by setting $V_0=V_3$ and $V_4=V_1$. We may refer to a copy of $D$ as a copy of the \emph{$i$-th type} for some $i\in[3]$ if such copy has two vertices in $V_i$ and one vertex in each of $V_{i-1},V_{i+1}$. 

The proof of Lemma~\ref{lem:extremal} will rely on the following result, which allows us to obtain a perfect tiling after cleaning up the extremal partition. 

\begin{theorem}[Theorem 1.2 in~\cite{martin-zhao}]\label{thm:blackbox}
    For every positive real number $\gamma$ and positive integer $h$ there exists $n_0 \in \N$ such that if $n\ge n_0$ and $n$ is divisible by $h$, then for every graph $G$ with a vertex partition $\{V_1,V_2,V_3\}$ such that $|V_1|=|V_2|=|V_3|=n$ and for every $i\in [3]$ and $v\in V_i$, $|N(v)\cap V_{i+1}|,|N(v)\cap V_{i-1}|\ge (2/3+\gamma) n$, $G$ contains a perfect $K_{h,h,h}$-tiling.
\end{theorem}

\begin{corollary}\label{cor:0mod4}
    There exist $c', \xi >0$ and $n_0\in \N$ so that the following holds. Suppose $n\ge n_0$ is a multiple of $4$, and $G$ is a near-tournament on $n$ vertices such that $\delta^0(G)\ge (1/2-c')n$. If there is a partition $\{V_1,V_2,V_3\}$ of $V(G)$ such that 
    \begin{equation} \label{eq:div}
        |V_1|\equiv |V_2|\equiv |V_3|\equiv 0\pmod 4
    \end{equation}
    and, for every $i\in [3]$ and $v\in V_i$, 
    \begin{equation} \label{eq:v}
        d^+(v,V_{i+1}), \; d^-(v,V_{i-1})\ge (1/3-\xi)n ,
    \end{equation}
    then $G$ has a perfect $D$-tiling.
\end{corollary}

\begin{proof}
    In what follows, we assume $c', \xi >0$ are sufficiently small and $n \in \N$ is sufficiently large so that all strict inequalities are true.  
    Without loss of generality, assume that $|V_1|\le |V_2|\le |V_3|$.
    For $i\in[3]$, the degree conditions guarantee that $(1/3 - \xi)n \le |V_i|\le (1/3 + 2\xi)n$ and 
    $$\delta^0(G[V_i])\ge 
    \left( \frac{1}{2} -c' \right)n - \left( \frac{1}{3} + 2\xi \right) n 
    \ge \left( \frac{1}{2} - 3c' - 9\xi \right) \left( \frac{1}{3} +2 \xi \right) n
    \ge \left( \frac{1}{2} - 3c' - 9\xi \right) |V_i|.$$ 
    We then set $c \coloneqq 3c'+9\xi$ and note that, for $i\in \{2,3\}$, 
    \begin{equation*}
        \left( \frac{1}{32}-4c \right)|V_{i}| \ge 
        \left( \frac{1}{32}-4c \right) \left( \frac{1}{3}-\xi \right)n > 
        3\xi n \ge |V_{i}|-|V_{1}|.
    \end{equation*}
    
    For $i\in \{2,3\}$, by Lemma~\ref{lem:minimum_degree} applied to $G[V_i]$, we can greedily find a collection $\D_i$ of $|\D_i| = (|V_i|-|V_{1}|)/4 \le \xi n$ vertex-disjoint copies of $D$ in $G[V_i]$.    

    We set $V_1' \coloneqq V_1$, $V_2' \coloneqq V_2 \setminus V(\D_2)$ and $V_3' \coloneqq V_3 \setminus V(\D_3)$ so that $|V_1'|=|V_2'|=|V_3'|$.  
    Let $G'$ be the (simple, non-oriented) $3$-partite graph on the vertex set $V(G') = V_1'\cup V_2' \cup V_3'$ and the edge set $E(G')$ corresponding to the underlying edges from $\bigcup_{i=1}^3 E^+_G(V_i', V_{i+1}')$. For every $i \in [3]$ and $v \in V_i'$, we have that 
    \begin{equation*}
        |N_{G'}(v) \cap V_{i+1}'| \ge d^+_G(v, V_{i+1}) - \xi n \ge 
        \left( \frac{1}{3}-2\xi \right) n > 
        \left( \frac{2}{3} + \xi \right) \left(\frac{1}{3}+2\xi\right)n \ge \left( \frac{2}{3} + \xi \right) |V_i'| ,
    \end{equation*}
    and, similarly, $|N_{G'}(v) \cap V_{i-1}'| \ge \left( \frac{2}{3} + \xi \right) |V_i'|$.
    By Theorem~\ref{thm:blackbox} applied to the graph $G'$ with $h=4$, and by Lemma~\ref{lem:k444}, there exists a perfect $D$-tiling $\D$ of the oriented graph induced by $V(G)\setminus V(\D_2\cup \D_3)$, and $\D\cup \D_2\cup \D_3$ is a perfect $D$-tiling of $G$.
\end{proof}

While Lemma~\ref{lem:minimum_degree} allows us to greedily find copies of $D$ covering the few vertices $v$ that do not satisfy~\eqref{eq:v}, we will also need the following lemma, which allows us to find copies of $D$ while guaranteeing the divisibility condition~\eqref{eq:div}. 

\begin{lemma}\label{lem:ith}
    Let $\alpha, \beta, c>0$ be such that $\alpha+2\beta+2c < 1/3$. Let $G$ be an oriented graph with $\delta^0(G)\ge (1/2-c)n$. Suppose that there are disjoint $V_1,V_2,V_3\subseteq V(G)$ satisfying that for every $i\in [3]$, we have $|V_i|\ge (1/3-\alpha)n$ and $V_i\subseteq N^-_\beta(V_{i+1})$ and $V_i\subseteq N^+_\beta(V_{i-1})$. Then, for every $u\in V_1\cup V_2\cup V_3$ and $i \in [3]$, there is a copy of $D$ in $G$ of $i$-th type that contains the vertex $u$.
\end{lemma}

\begin{proof}
    Without loss of generality, it suffices to show that an arbitrary vertex $u\in V_1$ is contained in copies of $D$ of any of the three types.
    By assumption, we have $d^+(u,V_{2})\ge |V_{2}|-\beta n$ and $d^-(u,V_{3})\ge |V_{3}|-\beta n$.
    We may pick some vertex $v\in V_2$ with $v\in N^+(u)$.
    Since
    $$|V_3\cap N^-(u)\cap N^+(v)|\ge d^-(u, V_3)+d^+(v, V_3)-|V_3|\ge |V_3|-2\beta n \ge \left( \frac{1}{3}-\alpha - 2\beta \right) n >2cn,$$
    there are two distinct vertex $z,w\in V_3\cap N^-(u)\cap N^+(v)$ that either $zw$ or $wz$ is an edge in $G$.
    Then $G[\{u,v,w,z\}]$ induces a copy of $D$ of the third type containing $u$.
    Similarly, we have that
    $$|V_1\cap N^+(w)\cap N^-(v)|\ge |V_1|-2\beta n
    \quad \text{ and } \quad
    |V_2\cap N^+(u)\cap N^-(w)|\ge |V_2|-2\beta n.$$
    For $x\in V_1\cap N^+(w)\cap N^-(v)\cap N(u)$ and $y\in V_2\cap N^+(u)\cap N^-(w)\cap N(v)$,
    we conclude that $G[\{u,v,w,x\}]$ contains a copy of $D$ of the first type containing $u$ and $G[\{u,v,w,y\}]$ contains a copy of $D$ of the second type containing $u$.
\end{proof}

We are now ready to prove Lemma~\ref{lem:extremal}.

\begin{proof}[Proof of Lemma~\ref{lem:extremal}]
    In what follows, we assume that $c, \gamma >0$ are sufficiently small so that there exists $\beta > \max\{ 4c, \sqrt{60\gamma} \}$ satisfying   
    \begin{equation}\label{eq:beta<c}
        \beta < \frac{1}{32}-4c
    \end{equation}
    \begin{equation}\label{eq:bgc<1/3}
        9\beta + \gamma + 2c < \frac{1}{3}
    \end{equation}
    and so that Corollary~\ref{cor:0mod4} holds with parameters $c' > c+3\beta$ and $\xi > 2\beta +\gamma $. 
    
    We fix a partition $\{V_1,V_2,V_3\}$ of $V(G)$ such that, for every $i\in[3]$, $(1/3 - \gamma)n \le |V_i| \le (1/3 + \gamma)n$ and $e^-(V_i,V_{i+1})\le \gamma n^2$.
    Let $U_i=N^-_\beta(V_{i+1})\cap N^+_\beta(V_{i-1}) \cap V_i$ for $i\in [3]$, and $U_0=V\setminus(U_1\cup U_2\cup U_3)$.
    For every $i\in[3]$ and $v\in V_i\setminus U_i$, by definition of $U_i$, at least one of the following holds:
    \begin{align*}
        d^-(v,V_{i+1})\ge |V_{i+1}|-(|V_{i+1}|-\beta n)-(n-2\delta^0(G))=\beta n-2cn > \beta n/2,\\
        d^+(v,V_{i-1})\ge |V_{i-1}|-(|V_{i-1}|-\beta n)-(n-2\delta^0(G))=\beta n-2cn > \beta n/2.
    \end{align*}
    Thus we have that
    \begin{equation}
        |U_0|=\sum_{i=1}^3|V_i\setminus U_i|\le \frac{\sum_{i=1}^3 2e^-(V_i,V_{i+1})}{\beta n/2}\le \frac{12\gamma n^2}{\beta n} < \frac{\beta n}{5}.\label{equ:u0}
    \end{equation}

    By~\eqref{eq:beta<c},~\eqref{equ:u0} and Lemma~\ref{lem:minimum_degree}, we can greedily find a collection $\C_0$ of vertex-disjoint copies of $D$ that covers all of $U_0$.
    For $i\in[3]$, let $W_i=U_i\setminus V(\C_0)$ and $W=W_1\cup W_2\cup W_3$.
    By~\eqref{equ:u0}, since $|V(\C_0)|\le 4|U_0|$, we have that $|V_i\setminus W_i|\le |V(\C_0)| + |U_0| < \beta n$. This implies that
    \begin{equation} \label{eq:|Wi|}
        |W_i| > |V_i| - \beta n \ge \left(\frac{1}{3} - \gamma - \beta \right) n, \text{ and}
    \end{equation}
    \begin{equation} \label{eq:deg_Wi}
        \delta^0(G[W]) > \left( \frac{1}{2} -c \right)n - 3\beta n \ge 
        \left( \frac{1}{2} -c - 3\beta \right) |W|.
    \end{equation}

    For $v \in W_i$, we have that $v \in N_\beta^-(V_{i+1})$, which implies $d^+(v, V_{i+1}) \ge |V_{i+1}|-\beta n$ and then $d^+(v, W_{i+1}) \ge |W_{i+1}|-\beta n$. Similarly, $d^-(v, W_{i-1}) \ge |W_{i-1}|-\beta n$. That is, 
    \begin{equation} \label{eq:Wi_subset}
        W_i \subseteq N_\beta^-(W_{i+1}) \cap N_\beta^+(W_{i-1}).
    \end{equation}
    
    Let the triple $(a_1,a_2,a_3)\in \{0,1,2,3\}^3$ be such that $|W_i| \equiv a_i \pmod 4$ for every $i \in [3]$. Note that $a_1+a_2+a_3 \equiv n \equiv 0 \pmod 4$. For every $i \in [3]$, set $b_i \equiv a_i - \frac{a_1+a_2+a_3}{4} \pmod 4$ so that $b_i \in \{0,1,2,3\}$.
    
    By~\eqref{eq:bgc<1/3},~\eqref{eq:|Wi|},~\eqref{eq:deg_Wi},~\eqref{eq:Wi_subset}, and Lemma~\ref{lem:ith} applied multiple times, there exists a collection $\C_1$ of disjoint copies of $D$ such that $\C_1$ contains exactly $b_i$ many copies of $D$ of the $i$-th type for every $i\in [3]$.
    It follows
    \begin{itemize}
        \item For $W_i'=W_i\setminus V(\C_1)$, we have $|W_i'| \equiv a_i - 2b_i - b_{i+1} -b_{i+2}\equiv 0\pmod 4$,
        \item $|V(\C_1)| = 4 (b_1+b_2+b_3) \le 36$.
    \end{itemize}

    Let $W'=W_1'\cup W_2'\cup W_3'$. We note that, for every $i\in [3]$ and $v \in W_i'$,
    \begin{equation}
        d^+(v, W_{i+1}') \ge d^+(v, W_{i+1}) - 12 \ge   
        \left( \frac{1}{3} - \gamma - 2\beta \right) n \ge
        \left( \frac{1}{3} - \xi \right) |W'|,
    \end{equation}
    and, similarly, $d^-(v, W_{i-1}')\ge \left( \frac{1}{3} - \xi \right) |W'|$.
    Then, by Corollary~\ref{cor:0mod4}, there is some perfect $D$-tiling $\C_2$ of $G[W']$, and $\C_0 \cup\C_1\cup \C_2$ is a perfect $D$-tiling of $G$.
\end{proof}

\section{Concluding remarks} \label{sec:conclusion}

\subsection{Tur\'anability}

In this paper, we show that there exists an oriented graph which is Tur\'anable and it is not a subgraph of $D_{s}$ for any $s \ge 2$, disproving Conjecture~\ref{old_conj}. However, the example we gave is a subgraph of the regular tournaments $F_r$ (for $r=2$) and $C_{2k+1}^k$ (for $k=2$), which suggests the following conjecture.

\setcounter{theorem}{7} \setcounter{section}{1}
\begin{conjecture}
    An oriented graph $H$ is Tur\'anable if and only if there are integers $r, k$ such that $H \subset F_r$ and $H \subset C_{2k+1}^k$.
\end{conjecture}
\setcounter{section}{6} \setcounter{theorem}{0}

In the context of edge-ordered graphs (see, e.g.,~\cite{APTX, GMNPTV}), there are four \emph{canonical edge-orders} so that an edge-ordered graphs is Tur\'anable if and only if it is a subgraph of all four canonical edge-orders. In this notation, Conjecture~\ref{new_conj} asserts that there is a finite list of canonical regular tournaments, and this list consists of the two regular tournaments $F_r$ and $C_{2k+1}^k$. We note that a counterexample to Conjecture~\ref{new_conj} would provide another semi-regular near-tournament to be included in this list. Even if Conjecture~\ref{new_conj} turns out to be false, we strongly believe that there should only be finitely many elements in a list of \emph{canonical semi-regular near-tournaments}.  

\subsection{Tileability}

We also show that the tournament $D_{1,1,2}$ is tileable, while $D_{s,s,s}$ is not tileable for any $s\ge 2$. It would be interesting to obtain a characterization of \emph{tournaments} that are Tur\'anable and not tileable. 
Since we know from Theorem~\ref{thm:bh} that a tournament is Tur\'anable if and only if it is subgraphs of $D_s$ for some $s\ge 1$, an answer to the following question would provide such characterization.

\begin{question} 
    For which $a\le b \le c$ is the oriented graph $D_{a,b,c}$ tileable?
\end{question}

We believe that $D_{a,b,c}$ will be tileable unless there is a divisibility barrier that does not break after changing the orientation of edges in a subgraph of bounded degree akin to the construction in Section~\ref{sec:ds}, where the bound on the degrees will depend on $a$, $b$, and $c$.

\subsection{Semi-degree thresholds}
In~\cite{DHLMPT}, it is shown that $C_{k}^\ell$ is Tur\'anable if and only if $k \ge 3\ell$. 
The authors of~\cite{DHLMPT} then ask what is $\kappa^0(C_{3t}^t)$.
In this paper, we show that $\kappa^0(C_{6}^2) \ge 3/7$.
It follows from Claim~\ref{claim:blowup} that it is enough to give a single example of an oriented graph without copy of $C_{3t}^t$ to obtain a lower bound on $\kappa^0(C_{3t}^t)$. 
Therefore, $\kappa^0(C_{3t}^t) \ge \lfloor \frac{3t-2}{2} \rfloor \cdot \frac{1}{3t-1}$ follows from the fact that a semi-regular near-tournament on $3t-1$ vertices does not contain an oriented graph on $3t$ vertices as a subgraph.
We ask the following questions.

\begin{question}
    Is it true that $\kappa^0(C_{3t}^t) > \lfloor \frac{3t-2}{2} \rfloor \cdot \frac{1}{3t-1}$ for every $t \ge 2$?
\end{question}

\begin{question}
    Is it true that $\kappa^0(C_{6}^2) > 3/7$?
\end{question}

\section*{Acknowledgments}

The authors thank Louis DeBiasio and the anonymous referees for their helpful comments and corrections on earlier versions of the paper.

\appendix
\section{Proof of Lemma~\ref{lem:merge}} \label{appendix}

Let $\beta>0$ be sufficiently small so that there exists $\xi>0$ satisfying 
$$ \sqrt{2 \beta} < \xi < \min \left\{ \frac{1}{(20\ell+3)!},  6 (\beta')^{8\ell+3} , \mu^2 (\beta')^{5} \right\} .$$

Then two vertices $u_0$ and $v_0$ are $(H_4,\beta,5\ell+1)$-reachable if there exist at least $ \xi n^{20\ell+3}$ ordered $(20\ell+3)$-tuples that can be permuted to $(u_1, u_2, \dots, u_{20\ell+3})$ and $(v_1, v_2, \dots, v_{20\ell+3})$ so that $\{u_{4j},u_{4j+1},u_{4j+2},u_{4j+3}\}$ and $\{v_{4j},v_{4j+1},v_{4j+2}, v_{4j+3}\}$ induce copies of $D$ for every $0\le j \le 5\ell$.

Indeed, this follows from $\beta< \xi^2/2$, the fact that there are at most $(20\ell+3)!<1/\xi$ orderings for each such tuple, and the number of tuples containing $u_0$, $v_0$ or repeated vertices is at most 
$$ 2(20\ell+3)n^{20\ell+2} + n \cdot (20\ell+3)^2 \cdot n^{20\ell+1} < \xi n^{20\ell+3}/2 ,$$
where the last inequality holds for sufficiently large $n\ge n_0$.

We will show that a pair of vertices $u, v$ is $(H_4,\beta,5\ell+1)$-reachable in both cases when $u,v$ belong to the same part of $\CP$ and when $u\in V_i$, $v \in V_j$ for $i,j \in [d]$ so that $\mathbf{u}_i - \mathbf{u}_j \in L_\CP^{\mu_4}(H_4)$.
    
\begin{claim}
    If $u,v$ are $(H_4,\beta',\ell)$-reachable, then they are $(H_4,\beta,5\ell+1)$-reachable.
\end{claim}

\begin{claimproof}
    Let $\CT$ be the set of ordered $(20\ell+3)$-tuples $(x_1, \dots, x_{20\ell+3}$ such that 
    \begin{itemize}
        \item $\{x_1,\dots,x_{4\ell-1}\}$ is $(H_4,u,v)$-linking, and 
        \item $G[\{x_{4j},x_{4j+1},x_{4j+2},x_{4j+3}\}]$ contains a copy of $D$ for every $\ell \le j\le 5\ell$.
    \end{itemize}

    There are $\beta'n^{4\ell-1}$ ways to choose the prefix of the tuple since $u,v$ are $(H_4,\beta',\ell)$-linking, and $(24d(V(G)))^{4\ell+1}$ ways to choose the remaining vertices.
    Within the partition $\CP$, pick some $V_q$ with $|V_q|\ge n/d \ge \beta' n$.
    Since $V_q$ is $(H_4,\beta',\ell)$-closed, each $v\in V_q$ is contained in at least $\beta'n^3$ many copies of $D$.
    Then, $24d(V(G))\ge (3!)(\beta' n)(\beta'n^3)=6(\beta')^2n^4$.
    We conclude that 
    $$|\CT|\ge \beta'n^{4\ell-1}\cdot (6(\beta')^2n^4)^{4\ell+1} > \xi n^{20\ell+3},$$
    which implies $u, v$ are $(H_4,\beta,5\ell+1)$-reachable.
\end{claimproof}

\begin{claim}
    If $\mathbf{u}_i - \mathbf{u}_j \in L_\CP^{\mu_4}(H_4)$, $u_0\in V_i$ and $v_4\in V_j$, then $u_0,v_4$ are $(H,\beta,5\ell+1)$-reachable.
\end{claim}

\begin{claimproof}
    Since $\mathbf{u}_i - \mathbf{u}_j \in L_\CP^{\mu_4}(H_4)$, there are $A,B,C$ such that $d(V_i,A,B,C)$ and $d(A,B,C,V_j)$ are both at least $\mu_4 n^4$.
    Let $\CT$ be the set of $(20\ell+3)$-tuples
    $$(v_0,v_1,v_2,v_3,u_1,u_2,u_3,u_4,w_1,\dots,w_{20\ell-5})$$ satisfying
    \begin{itemize}
        \item $G[v_0,v_1,v_2,v_3]$ is a copy of $D$ with $v_0\in V_i,v_1\in A, v_2\in B, v_3\in C$.\;
        \item $G[u_1,u_2,u_3,u_4]$ is a copy of $D$ with $u_1\in A, u_2\in B, u_3\in C,u_4\in V_j$.\;
        \item $\{w_{i(4\ell-1)+1},\dots,w_{(i+1)(4\ell-1)}\}$ is $(H_4,u_i,v_i)$-linking $(4\ell-1)$ set for $i\in\{0,1,2,3,4\}$.
    \end{itemize}
    We may count $|\CT|\ge (\mu n^4)^2\cdot (\beta' n^{4\ell-1})^5\ge \xi n^{20\ell +3}$, so $u_0$ and $v_4$ are $(H_4,\beta,5\ell+1)$-reachable.
\end{claimproof}

\end{document}